\documentclass[10pt]{article}

\usepackage{amsmath,amsfonts,amssymb} % for math

\usepackage{amsthm} % theorems

\usepackage{graphicx,caption,subcaption} % figures and subfigures
\usepackage{fullpage}

\allowdisplaybreaks % otherwise the align environment is stupid.

\begingroup
    \makeatletter
    \@for\theoremstyle:=definition,remark,plain\do{%
        \expandafter\g@addto@macro\csname th@\theoremstyle\endcsname{%
            \addtolength\thm@preskip\parskip
            }%
        }
\endgroup

\newtheorem{thm}{Theorem}[section]
\newtheorem{cor}[thm]{Corollary}

\newtheorem{lemma}[thm]{Lemma}

\newtheorem{definition}[thm]{Definition}

\newtheorem{claim}[thm]{Claim}

\title{}
\author{}

\begin{document}

\title{\vspace{-0.5in} Dirac's Theorem for hamiltonian Berge cycles in uniform hypergraphs}

\author{
{{Alexandr Kostochka}}\thanks{
\footnotesize {University of Illinois at Urbana--Champaign, Urbana, IL 61801
 and Sobolev Institute of Mathematics, Novosibirsk 630090, Russia. E-mail: \texttt {kostochk@math.uiuc.edu}.
 Research %%% of this author
is supported in part by  NSF RTG Grant DMS-1937241 and grant 19-01-00682  of the Russian
Foundation for Basic Research.
}}
\and
{{Ruth Luo}}\thanks{
\footnotesize {University of California, San Diego, La Jolla, CA 92093. E-mail: \texttt {ruluo@ucsd.edu}.
 Research %%% of this author
is supported in part by NSF grant DMS-1902808.
}}
\and{{Grace McCourt}}\thanks{University of Illinois at Urbana--Champaign, Urbana, IL 61801, USA. E-mail: {\tt mccourt4@illinois.edu}. Research %%% of this author
is supported in part by NSF RTG grant DMS-1937241.}}

\date{ \today}
\maketitle

\vspace{-0.3in}

\begin{abstract}
The famous Dirac's Theorem gives an exact bound on the minimum degree of an $n$-vertex graph guaranteeing the existence of a hamiltonian cycle. We prove  exact bounds of similar type for hamiltonian Berge cycles in $r$-uniform, $n$-vertex hypergraphs for all $3\leq r< n$. The bounds are different for $r<n/2$ and $r\geq n/2$. We also give bounds  on the minimum degree guaranteeing existence of Berge cycles of length at least $k$ in such hypergraphs; the bounds are exact for all $k\geq n/2$.

\end{abstract}

\section{Introduction and Results}
%{\bf Terminology.} 
\subsection{Terminology and known results}
 A hypergraph $H$ is a family of subsets of a ground set. We refer to these subsets as the {\em edges} of $H$ and the elements of the ground set as the {\em vertices} of $H$. We use $E(H)$ and $V(H)$ to denote the set of edges and the set of vertices of $H$ respectively. We say $H$ is {\em $r$-uniform}  ({\em an $r$-graph}, for short) if every edge of $H$ contains exactly $r$ vertices. A {\em graph} is a 2-graph.

The {\em degree} $d_H(v)$ of a vertex $v$ in a hypergraph $H$ is the number of edges containing $v$. The {\em minimum degree}, $\delta(H)$, is the minimum over degrees of all vertices of $H$. The {\em circumference}, $c(G)$, is the length of a longest cycle in $G$.

A {\em hamiltonian cycle} in a graph is a cycle that visits every vertex. %We say a graph is {\em hamiltonian} if it contains a hamiltonian cycle. 
 Sufficient conditions for existence of hamiltonian cycles in graphs have been well-studied. In particular, the first extremal result of this type was due to Dirac in 1952. 

\begin{thm}[Dirac~\cite{D}]\label{dirac}
Let $n \geq 3$. If $G$ is an $n$-vertex graph with  $\delta(G) \geq n/2$, then $G$ has a hamiltonian cycle.
\end{thm}

Dirac also proved that $c(G)\geq \delta(G)+1$ for every graph $G$.
%Recently, a series of results of the same flavor was proved  for 
We consider similar conditions for
{\em Berge cycles} in hypergraphs.

\begin{definition}
A {\bf Berge cycle of} length $\ell$ in a hypergraph is a list of $\ell$ distinct vertices and $\ell$ distinct edges $v_1, e_1, v_2, \ldots,e_{\ell-1}, v_\ell, e_\ell, v_1$ such that $\{v_i, v_{i+1}\} \subseteq e_i$ for all $1\leq i \leq \ell$ (here we take indices modulo $\ell$).
Similarly, a {\bf Berge path} of length $\ell$ is a list of $\ell+1$ distinct vertices and $\ell$ distinct edges $v_1, e_1, v_2, \ldots, e_{\ell}, v_{\ell + 1}$ such that $\{v_i, v_{i+1}\} \subseteq e_i$ for all $1\leq i \leq \ell$. 
\end{definition}

Although the edges in a Berge cycle may contain other vertices, we say $V(C)=\{v_1, \ldots, v_\ell\}$, and $E(C) = \{e_1, \ldots, e_\ell\}$. 
Notation for  Berge paths is similar. We use $c(H)$ to denote the circumference of a hypergraph $H$, that is, the length of a longest Berge cycle in $H$.

An analogue of Dirac's Theorem for non-uniform hypergraphs was given in~\cite{FKLdirac}. 
For $r$-uniform hypergraphs, a well-known approximation of Dirac's bound on circumference and of Theorem~\ref{dirac}
was proved by Bermond,  Germa,  Heydemann and Sotteau~\cite{BGHS} more than 40 years ago:

\begin{thm}[Bermond, et al.~\cite{BGHS}]\label{bermond}Let $r \geq 3$ and $k \geq r+1$. If $H$ is an $r$-uniform hypergraph with $\delta(H) \geq {k-2 \choose r-1} + r - 1$, then $H$ contains a Berge cycle of length $k$ or longer. In particular, if  $\delta(H) \geq {n-2 \choose r-1} + r - 1$, then $H$ contains a hamiltonian Berge cycle.
\end{thm}

Recently, there was a series of improvements of the hamiltonian part of Theorem~\ref{bermond}.
First, Clemens, Ehrenm\"uller and Person~\cite{CEP} have proved an asymptotics for $n>2r-2$:

\begin{thm}[Clemens et al.~\cite{CEP}]\label{cepthm}
If $n > 2r-2$ and $\delta(H) \geq {\lfloor (n-1)/2 \rfloor \choose r-1} + n-1$, then $H$ has a hamiltonian Berge cycle.
\end{thm}

Then Coulson and Perarnau~\cite{CP} proved the exact bound for  $n$ much larger than $r$:

\begin{thm}[Coulson and Perarnau~\cite{CP}]\label{cpthm} Let $H$ be an $r$-graph on $n$ vertices such that $r = o(\sqrt{n})$. If  $\delta(H) \geq {\lfloor (n-1)/2 \rfloor \choose r-1} + 1$, then $H$ contains a hamiltonian Berge cycle.
\end{thm}

%This bound in Theorem~\ref{cpthm} on $\delta(H)$ is sharp. Moreover, Coulson and Perarnau's result  is a corollary of their stronger result for rainbow cycles in graphs. With their methods, $r = o(\sqrt{n})$ is best possible, so the case $r = \Omega(\sqrt{n})$ remained open. Various results by Bermond, Germa, Heydemann, and Sotteau~\cite{BGHS}; Clemens, Ehrenm\"uller, and Person~\cite{CEP}; and Ma, Hou, and Gao~\cite{MHG} provided upper bounds for the minimum degree when $r = \Omega(\sqrt{n})$. The best bound to date was due to Ma, Hou and Gao. 

Then Ma, Hou and Gao~\cite{MHG} improved the bound of Theorem~\ref{cepthm} for $n \geq 2r+4$.

\begin{thm}[Ma, Hou and Gao~\cite{MHG}]\label{MHGthm}Let $r \geq 4$ and $n \geq 2r+4$, and let $H$ be an $r$-graph on $n$ vertices. If  $\delta(H) \geq {\lfloor (n-1)/2 \rfloor \choose r-1} + \lceil (n-1)/2 \rceil$, then $H$ contains a hamiltonian Berge cycle.
\end{thm}

\subsection{Our results}
In this paper  we derive exact bounds for all possible $3\leq r<n$, improving the aforementioned theorems.

%In~\cite{main2}, we prove that the bound $\delta(H) \geq \max\{\lfloor (n-1)/2 \rfloor+1, r\}$ is sufficient for all $r \geq \lfloor (n-1)/2 \rfloor - 1$. In this paper, we prove a sharp bound for $r < \lfloor (n-1)/2 \rfloor$. 

\begin{thm}\label{main}
Let $t=t(n)=\lfloor (n-1)/2 \rfloor$, and suppose $3 \leq  r <n$. Let $H$ be an $r$-graph. If

\vspace{-3mm}

\begin{enumerate}
 \item[(a)] $r\leq t$ and $\delta(H) \geq  {t \choose r-1} +1$ or 

\vspace{-2mm}

\item[(b)] $r\geq n/2$ and $\delta(H) \geq r$, 
 \end{enumerate}
 
\vspace{-3mm}

 then $H$ contains a hamiltonian Berge cycle. 
\end{thm}

These bounds are best possible due to the following constructions.

{\bf Construction 1}. Suppose $r \leq t$. If $n$ is odd, let $H_1$ consist of two copies of $K^{(r)}_{(n+1)/2}$ that share exactly one vertex. If $n$ is even, let $H_1$ consist of two disjoint $K^r_{n/2}$ and a single edge intersecting both  cliques. 

{\bf Construction 2}. Suppose $r \leq t$. Let $H_2$ have vertex set $X \cup Y$ such that  $|X| =t$ and $|Y| = n-t$. The edge set of $H_2$ consists of every edge with at most one vertex in $Y$.

{\bf Construction 3}. Suppose $r \geq n/2$. Let $H_3$ be obtained by removing a single edge from an $r$-uniform tight cycle on $n$ vertices.

It is easy to check that both $H_1$ and $H_2$ have minimum degree ${t \choose r-1}$. Observe that neither $H_1$ nor $H_2$ have a hamiltonian Berge cycle: $H_1$ has  either a cut vertex or a cut edge, and in $H_2$ a hamiltonian Berge cycle must visit two vertices in $Y$ consecutively, but no edge of $H_2$ contains any pair of vertices from $Y$. 
%Similarly, if $n$ is even, since $|Y| = |X| + 2$, there must be two instances in which vertices of $Y$ are visited consecutively, but only 1 edge contains multiple vertices from $Y$.

Since an $r$-uniform tight cycle is $r$-regular, $\delta(H_3) =r-1$. Also, $H_3$ does not have a hamiltonian Berge cycle because $|E(H_3)| = n-1$. In fact, removing a single edge from {\em any} $r$-regular, $r$-uniform, $n$-vertex hypergraph would also yield an extremal example.
%Any $r$-graph $H_3$ has minimum degree $r-1$ by definition and has not enough edges for a hamiltonian cycle.

Note that the length of the longest cycle in Construction 1 is $\lceil n/2\rceil$. Thus Theorem~\ref{main} yields exact bounds on the minimum degree guaranteeing the existence of any cycle of length at least  $k$  in $n$-vertex $r$-uniform hypergraphs for all $r\leq t$ and all 
$k\geq 1+n/2$.

%Generalizing the problem for hypergraphs with bounded circumference, as an extention of the Erd\H{o}s--Gallai Theorem~\cite{EG}, Bermond, et al. proved the following result for hypergraphs. 

%\begin{thm}[Bermond, et al.~\cite{BGHS}]\label{bermond}Let $r \geq 3$ and $k \geq r+1$. If $H$ is an $r$-uniform hypergraph with minimum degree $\delta(H) \geq {k-2 \choose r-1} + r-1$, then $H$ contains a Berge cycle of length $k$ or longer.
%\end{thm}

We also improve the circumference part of Theorem~\ref{bermond}. Since the bounds for $r\leq t$ and for $r>t$ are different, we state our results as two theorems. %, and discuss sharpness examples after each of them. 
%As above, we set $t = t(n) = \lfloor(n-1)/2\rfloor$.

 %which is best possible for infinitely many $n$.

\begin{thm}\label{main3}Let $n, k,$ and $r$ be positive integers such that $n \geq k$ and $t \geq r \geq 3$. Let $H$ be an $n$-vertex, $r$-uniform hypergraph. If 

\vspace{-3mm}

\begin{enumerate}
\item[(a)] $k \leq r+1$ and $\delta(H) \geq k-1$, or

\vspace{-2mm}

\item[(b)] $r+2\leq k<t+2$ and $\delta(H) \geq {k-2 \choose r-1}+1$, or

\vspace{-2mm}

\item[(c)] $k\geq t+2$ and $\delta(H) \geq {t \choose r-1}  + 1$,
\end{enumerate}
\vspace{-3mm}
then $H$ contains a Berge cycle of length $k$ or longer.
\end{thm}

\begin{thm}\label{main4}
Let $n$, $k$, and $r$ be positive integers such that $n \geq k \geq r\geq 3$, and $r > t$. If $H$ is an $r$-uniform hypergraph %on at least $k$ edges 
 with \[\delta(H) \geq \lfloor \frac{r(k-1)}{n} \rfloor+1,\] then $H$ contains a Berge cycle of length $k$ or longer. 
%\[\delta(H) \geq \left\{\begin{array}{ll}  \min\{\lfloor \frac{n(k-1)}{r} \rfloor+1,r\} & \mbox{when } r> t,\\
%r+1&  \mbox{when } r = t,\end{array}\right.\] then $c(H)\geq k$. % contains a Berge cycle of length $k$ or longer. 
\end{thm}

%When $k-2 \geq t$, i.e., $k \geq \lfloor (n+1)/2 \rfloor + 1$, 
 Constructions 1 and 2 give sharpness examples for
Theorem~\ref{main3}(c). The constructions below show that for each $k\geq 3$ the bounds of Theorem~\ref{main3}(a,b) are sharp for  infinitely many $n$.

{\bf Construction 4}. Let $r+2\leq k < t+2$. For $n-1$  divisible by $k-2$,  let $H_4$ consist of $(n-1)/(k-2)$ copies of $K^{(r)}_{k-1}$ such that all the cliques share exactly one vertex. 

{\bf Construction 5}. Let $ k \leq r+1\leq t+1$.  For $n-1$  divisible by $r$,  view $V(H_5)$ as the union of $(n-1)/r$ sets $S_1,\ldots,S_{(n-1)/r}$ of $(r+1)$ vertices, all sharing exactly one vertex. The set $E(H_5)$ has arbitrary $k-1$
edges contained in each $S_i$.

We have $\delta(H_4) = {k-2 \choose r-1}$ and $\delta(H_5) = k-2$. A longest Berge cycle in $H_4$ must be contained in a single clique, and hence has length $k-1$. Similarly, a longest Berge cycle in $H_5$ is contained in some $S_i$, and hence has at most $k-1$ edges.

For Theorem~\ref{main4}, it is easy to construct an  analog of Construction 3: an $n$-vertex $r$-uniform hypergraph with $k-1$ edges whose minimum degree is exactly $\lfloor \frac{r(k-1)}{n} \rfloor$. 
%And for $r=t$, hypergraph $H_1$ in Construction~1 has minimum degree $r$ and has $c(H_1)=\lfloor (n+1)/2\rfloor$.

\subsection{Outline of the proofs}
As always, $t =t(n) = \lfloor (n-1)/2 \rfloor$. Together, the  circumference results, Theorem~\ref{main3} and Theorem~\ref{main4}, imply the hamiltonian result Theorem~\ref{main} by setting $k=n$. 

First we will prove Parts (a) and (b) of Theorem~\ref{main3}. Then we handle Part (c):
 for large $k$, our minimum degree condition  guarantees the existence  not only of a ``long" Berge cycle, but rather of a hamiltonian Berge cycle.

%\begin{thm}\label{main31}
%Let $n, k,$ and $r$ be positive integers such that $n \geq k$ and $t\geq r \geq 3$. Let $H$ be an $n$-vertex, $r$-uniform hypergraph. If $k\geq r+2$ and $\delta(H) \geq {k-2 \choose r}  + 1$, or if $k \leq r+1$ and $\delta(H) \geq k-1$, then $H$ contains a Berge cycle of length $k$ or longer.
%\end{thm}

%\begin{thm}\label{main32}
%Let $n$ and $r$ be positive integers such that $t \geq r$. Let $H$ be an $n$-vertex, $r$-uniform hypergraph. If $\delta(H) \geq {t \choose r}  + 1$, then $H$ has a hamiltonian Berge cycle.
%\end{thm}

\medskip
%Note that the $k=n$ case of Theorem~\ref{main4} implies Theorem~\ref{main} for $r \geq t$. In turn, Theorem~\ref{main} together with Construction~1 implies the case $r=t$ of Theorem~\ref{main4}. So, we will prove Theorem~\ref{main} for $r\leq t$. 

Observe that the inequality $\delta(H) \geq\lfloor \frac{r(k-1)}{n} \rfloor+1$ provides that $H$ has at least $k$ edges.
Hence the following theorem implies 
 Theorem~\ref{main4}.  % by instead proving a stronger result which easily implies it.

\begin{thm}\label{main41}
Let $n$, $k$, and $r$ be positive integers such that $n \geq k \geq r > t$ and $r \geq 3$. If $H$ is an $r$-uniform hypergraph with at least $k$ edges such that
$\delta(H) \geq \lceil \frac{k}{2} \rceil$, then $c(H)\geq k$. % contains a Berge cycle of length $k$ or longer. 
\end{thm}

So, we will prove Theorem~\ref{main41}.

In Section 2, we prove Theorem~\ref{main3}(a,b). In Section 3 we describe the setup of the proofs of Theorems~\ref{main3}(c) and~\ref{main41}. 
%In fact, instead of the case $r>t$ in these theorems we will prove a slight sharpening of Theorem~\ref{main4} stated in Section 3.
The proofs somewhat differ for $r<t$, $r=t$ and $r>t$. But in all cases we will use the same structure of proofs, namely, a modification of Dirac's original proof of his theorem. 

%We describe our
%The proofs for $r<t$ and for $r\geq t$ are quite different. For hamlitonian Berge cycles when $r<t$ we use the original approach of Dirac using lollipops. In Section 2 we set up the proof for this case and prove that our $r$-graph must have a cycle of length at least $t+2$. In Section 3 we analyze lollipops with a long (but not hamiltonian) Berge cycle, and in all cases get a contradiction. In Section 4 we prove Theorem~\ref{main3}. In Section 5, we prove the main lemmas we need for the proof of Theorem~\ref{main4}. In the remaining two sections, we prove the $r>t$ and $r=t$ cases of Theorem~\ref{main4}. The case $r=t$ is especially difficult and adds substantially to the length of the sections in which we consider $r\geq t$.

Also, since we always consider only Berge paths and cycles, from now on we drop the word ``Berge" and use cycles and paths to exclusively refer to Berge cycles and Berge paths.

\section{Proof of Theorem~\ref{main3}(a,b)}

We will use the following results.

\begin{thm}[\cite{KL1}]\label{EG}
Let $k \geq 4, r \geq k+1$, and let $ H $ be an $n$-vertex $r$-graph with no Berge cycles of length $k$ or longer. Then $e( H ) \leq \frac{(k-1)(n-1)}{r}$. %Furthermore, equality holds if and only if $ H $ is an $(r+1,k-1)$-block-tree.  
\end{thm}

\begin{thm}[Ergemlidze, Gy\H{o}ri, Methuku, Salia, Thompkins, and Zamora~\cite{EGMSTZ}]\label{EGr1}
Let $n\geq r \geq 3$, $k\in \{r+1,r+2\}$, and let $ H $ be an $n$-vertex $r$-graph with no Berge cycles of length $k$ or longer. Then $e( H ) \leq \frac{(k-1)(n-1)}{r}$.  
\end{thm}

\begin{proof}[Proof of Theorem~\ref{main3}(a).] %We consider two cases based on the length $k$ of the cycle.

% We first consider the the case of short cycles, that is $k \leq r+2$.
% 
% \begin{thm}\label{short}
% Let $n,r$ and $k$ be integers satisfying $n \geq r\geq 3$ and $3\leq k \leq  \min \{r+2,n\}$. Suppose $H$ is an $n$-vertex, $r$-uniform hypergraph with minimum degree $\delta(H) \geq k- 1$. Then $H$ contains a Berge cycle of length at least $k$.
% \end{thm}

%{\bf Case 1}: 
Recall that $3\leq k \leq  \min \{r+2,n\}$ and $\delta(H) \geq k- 1$.

By Theorems~\ref{EG} and~\ref{EGr1}, if $4\leq k\leq r-1$ or $r\geq 3$ and $k\in \{r+1,r+2\}$, then $e( H ) \leq \frac{(k-1)(n-1)}{r}$.  
It follows that the average degree of $ H $ is at most 
$$\frac{r}{n} \cdot \frac{(k-1)(n-1)}{r}=\frac{(k-1)(n-1)}{n}<k-1.$$
This gives that $ H $ has a vertex of degree at most $k-2$, a contradiction.

Thus to prove the theorem, we need to settle the remaining cases, namely,  $k=3\leq r$ and  $k=r\geq 4$.
In both cases, consider a counter-example $H$ with the most edges. Then $H$ contains a path of length at least $k-1$.  
Among all such paths, let $P = v_1, e_1, v_2, \ldots, e_{\ell-1}, v_\ell$ be a longest one. 

If there exists a $j \geq k$ such that $v_1 \in e_j$, then $v_1,e_1, v_2, \ldots,e_{j-1}, v_j, e_j, v_1$ is a cycle of length at least $k$. Furthermore, if there exists an edge $e \in E(H) - E(P)$ and a vertex $u \in V(H) - \{v_1, \ldots, v_{k-1}\}$ such that $\{v_1, u\} \subset e$, then either $u \notin V(P)$ and we can extend $P$ to a longer path by adding the vertex $u$ and the edge $e$, or $u \in V(P)$ and we can construct a cycle of length at least $k$ by combining the segment of $P$ from $v_1$ to $u$ with the edge $e$. 
Therefore each edge of $H$ containing $v_1$ either is in $\{e_1, e_2, \dots, e_{k-1}\}$ or is contained in $\{v_1, \ldots, v_{k-1}\}$.
Since $k-1<r$, the latter is impossible. Thus adding the fact that $d(v_1)\geq k-1$, we have that
\begin{equation}\label{au51}
\mbox{\em all edges $e_1,\ldots,e_{k-1}$ contain $v_1$.
}
\end{equation}
Since $H$ has no multiple edges, there is a vertex $v'\in e_1-e_{k-1}$. If $v'\notin \{v_1,\ldots,v_\ell\}$, then we 
consider path $P'$ obtained from $P$ by replacing $v_1$ with $v'$ and keeping all the edges.
It has the same length  as $P$, but $v'\notin e_{k-1}$, contradicting~\eqref{au51}. 

So, suppose $v'=v_j$. Since $v'\notin e_{k-1}$ and $v_1\in e_{k-1}$, $j\notin \{1,k-1,k\}$. If $j\geq k+1$, then we have a cycle
$C_2 = v_2,e_2, v_3, \ldots,e_{j-1}, v_j, e_1, v_2$, a contradiction. Thus $2\leq j\leq k-2$. Consider path 
$$P'' = v_j,e_{j-1}, v_{j-1}, \ldots, e_1, v_1, e_j, v_{j+1},e_{j+1}, v_{j+2},\ldots,e_{\ell-1}, v_\ell.$$
Similarly to $P'$, it has the same length  as $P$, but $v'\notin e_{k-1}$, contradicting~\eqref{au51}. 
\end{proof}

% We now consider finding long cycles, that is $k \geq r+3$.

% \begin{thm}\label{long}
% Let $n \geq k \geq  r+3\geq 6$ be integers. Suppose $H$ is an $n$-vertex, $r$-uniform hypergraph with minimum degree $\delta(H) \geq {k-2 \choose r-1} + 1$. Then $H$ contains a Berge cycle of length at least $k$.
% \end{thm}
\begin{proof}[Proof of Theorem~\ref{main3}(b).] Recall that
 $k \geq r+3$ and $\delta(H) \geq {k-2 \choose r-1} + 1$.
Suppose the theorem fails, and let $H$ be an edge-maximal counterexample. Then $H$ contains a path of length $k-1$ or greater. Among all such  paths, let $P = v_1,e_1, v_2, \ldots, e_{\ell-1}, v_\ell$ be a longest one. As in the 
proof of Theorem~\ref{main3}(a), each edge of $H$ containing $v_1$ either is in $E(P)$ or is a subset of $\{v_1, \ldots, v_{k-1}\}$.

%If there exists a $j \geq k$ such that $v_1 \in e_j$, then $(v_1,e_1, v_2, \ldots,e_{j-1}, v_j, e_j, v_1)$ is a cycle of length at least $k$. Furthermore, if there exists an edge $e \in E(H) - E(P)$ and a vertex $u \in V(H) - \{v_1, \ldots, v_{k-1}\}$ such that $\{v_1, u\} \subset e$, then either $u \notin V(P)$ and we can extend $P$ to a longer path by adding the vertex $u$ and the edge $e$, or $u \in V(P)$ and we can construct a cycle of length at least $k$ by combining the segment of $P$ from $v_1$ to $u$ with the edge $e$. 

Set $X = \{v_1, \ldots, v_{k-1}\}$ and $X'=X-v_1$. Let $E_X = \{e \notin E(P): e \subseteq X \}$. The previous paragraph implies that every edge containing $v_1$ belongs to $E_X \cup \{e_1, \ldots, e_{k-1}\}$.

{\bf Case 1}: There exists some $1 \leq i \leq k-2$ such that  $v_1\in e_i$ and $e_i \not\subseteq X $. Let $u \in e_i - X$.

%Thus $r\geq 4$.
 If there exists an edge $f \in E_X$ such that $\{v_1, v_{i+1}\} \subset f$, then 
\[u, e_i, v_i, e_{i-1}, v_{i-1}, \ldots,e_1, v_1, f, v_{i+1}, e_{i+1}, v_{i+2}, \ldots, e_{\ell-1}, v_\ell\]
is longer than $P$, a contradiction. So, there are no such edges.

If $r=3$, then $i=1$, since otherwise $\{v_1,v_i,v_{i+1}\}\subset e_i$, and there is no room for other vertices in $e_i$. Therefore
\begin{equation}\label{p1}
d_H(v_1) \leq {|X' - \{v_{2}\}| \choose r-1} + |\{e_1, e_2, e_{k-1}\}| ={k-3 \choose r-1}+3 \leq  {k-2 \choose r-1},
\end{equation}
when $k \geq 6$, a contradiction.
%If both  $e_2$ and $e_{k-1}$ contain $v_1$, then $v_2\notin e_{k-1}$. In this case, consider path
%\[P_1=( v_2,e_1, v_{1}, e_2, v_3, e_3, v_4, \ldots,  e_{\ell-1}, v_\ell).\]
%It has the same length as $P$, and Case 1 holds for it, since  $e_1$ is not contained in $X$. However, $v_2\notin e_{k-1}$, so~\eqref{p1} will hold for $P_1$, a contradiction.

Suppose now that $r\geq 4$.
The number of edges in $E_X$ containing $v_1$ is at most ${|X' - \{v_{i+1}\}| \choose r-1} = {k-3 \choose r-1}$. 
Since $k\geq r+3$,  $k\geq 7$ and ${k-3 \choose r-2}\geq {k-3 \choose 2}\geq k-1$. So, 
$$d_H(v_1) \leq  {k-3 \choose r-1}+k-1= {k-2 \choose r-1}-{k-3 \choose r-2}+k-1\leq {k-2 \choose r-1},$$
 a contradiction.

{\bf Case 2}: For all $1 \leq i \leq k-2$ with $v_1\in e_i$, $e_i \subset X $. Then the only possible edge containing $v_1$ that is not a subset of 
$X $ is $e_{k-1}$, and $d_H(v_1) \leq {|X'| \choose r-1} = {k-2 \choose r-1}+1$ with equality if and only if $v_1 \in e_{k-1}$ and every $r$-subset of $X \cup \{v_1\}$ containing $v_1$ is an edge of $H$. 
Hence we may suppose this is the case. %In particular

For each $2\leq i\leq k-1$, let $g_i$ be the $(r-1)$-subset of $X'$ containing $v_i$ and the $r-2$ previous vertices of $X'$(with wrap around). I.e., if $i\geq r$, then $g_i = \{v_i, v_{i-1}, \ldots, v_{i-(r-2)}\}$ and if $i \leq r-1$, then $g_i = \{v_i, v_{i-1}, \ldots, v_{2}\} \cup \{v_{k-1}, \ldots, v_{k-1 - (r-1 - i)}\}$. Then set $f_i = g_i \cup \{v_1\}$. 
Since $k \geq  r+3$ and $\{v_1, v_{k-1}, v_k\} \subset e_{k-1}$, there exists some $2\leq i \leq k-2$ such that $v_i \notin e_{k-1}$. Then using the fact that $f_j \in E(H)$ for all $2\leq j \leq k-1$, the path
$$P_2=v_i, f_i, v_{i-1}, \ldots,f_2, v_1, f_{i+1}, v_{i+1}, f_{i+2}, v_{i+2}, \ldots,f_{k-1}, v_{k-1}, e_{k-1}, v_k, \ldots, e_{\ell-1}, v_\ell$$
is also a longest path. Applying the same argument to its first vertex $v_i$, we have that either $d_H(v_i) \leq {k-2 \choose r-1}$ or $v_i \in e_{k-1}$. In both cases we obtain a contradiction.
\end{proof}

\section{Setup of proofs for Theorems~\ref{main3}(c) and~\ref{main41} and general lemmas}

The original proof by Dirac of Theorem~\ref{dirac} involved two steps.
In the first step, by looking at a longest path, he greedily found a cycle  {of length at least $1+n/2$.} In the second step, he considered a
 {\em lollipop}, i.e. a pair $(C, P)$ such that $C$ is a cycle, $P$ is a path, $E(C) \cap E(P) = \emptyset$, $|V(C) \cap V(P)| = 1$, and the shared vertex of $v\in V(C)\cap V(P)$ is one of the endpoints of $P$. Dirac proved that when $\delta(G)\geq n/2$, the lollipop with the largest $|C|$ and modulo this with the largest $|P|$ can be only a hamiltonian cycle.
 
 Our strategy is in the same spirit, only instead of lollipops we will consider pairs of a cycle $C$  and a disjoint from $C$ path $P$. We will in addition maximize a couple of more parameters.

A pair $(C, P)$ of a cycle $C$ and a disjoint from $C$ path $P$ {\em is better} than a similar pair $(C',P')$ if 
\begin{enumerate}
    \item[(i)] $|E(C)|>|E(C')|$, or
    \item[(ii)] $|E(C)|=|E(C')|$ and $|E(P)|>|E(P')|$,
or
    \item[(iii)] $|E(C)|=|E(C')|$, $|E(P)|=|E(P')|$ and the total number of vertices in $V(P)$ in the edges in $C$ (counted with multiplicities) is greater than the total number of vertices in $V(P')$ in the edges in $C'$, or
    \item[(iv)] all parameters above coincide and    the total number of vertices in $V(P)$ in the edges in $P$ (counted with multiplicities) is greater than the total number of vertices in $V(P')$ in the edges in $P'$.
\end{enumerate}

Similarly to Dirac's proof, we will prove that in all cases, a best pair is a hamiltonian cycle (or contains a cycle of length at least $k$  when we are looking for such cycles).
 
% After proving a couple of auxiliary lemmas in the next section, in
In  all cases there will be 3 steps: first we find a cycle of length at least $1+n/2$, then prove that in the best pair $(C,P)$, $P$ cannot have only one vertex (unless $C$ is as long as we want), and finally show that $P$ also cannot have more than one vertex.

\subsection{General lemmas}
Suppose $(C,P)$ is a best pair with $C = v_1, e_1, \ldots, v_s, e_s, v_1$ and $P = u_1, f_1, \ldots, f_{\ell-1}, u_{\ell}$. 

We consider three subhypergraphs, $H_C,H_P$ and $H'$ of $H$ with the same vertex set $V(H)$:  $E(H_C)=\{e_1,\ldots,e_s\}$, $E(H_P)=\{f_1,\ldots,f_{\ell-1}\}$ and  $E(H')=E(H)-E(H_C)-E(H_P)$. 
Observe that the edges of these three subhypergraphs form a partition of the edges of $H$.
For a hypergraph $F$ and a vertex $u$, We denote by $N_{F}(u) = \{v \in V(F): \{u,v\} \subseteq e \text{ for some } e \in F\}$. 
For $i \in \{1, \ell\}$, set $B_i=\{e_j \in E(C): u_i \in e_j\}$.

The following claim applies to all best pairs $(C,P)$, regardless of the sizes of $r$ and $k$. It will be used in the sections below.

\begin{claim}\label{noconsecutive}
In a best pair $(C,P)$, $N_{H'}(u_1)$ cannot contain a pair of vertices that are consecutive in $C$.
\end{claim}

\begin{proof}
Suppose toward a contradiction that $v_i, v_{i+1}$ are contained in edges of $H'$ with $u_1$. Let $e, e' \in E(H')$ be such that $u_1, v_i \in e$ and $u_1, v_{i+1} \in e'$. If $e \neq e'$, then replacing $e_i$ with $e,u_1, e'$ gives a longer cycle than $C$, a contradiction. Thus we may assume $e = e'$.

If there is $1 \leq j \leq \ell$ such that $u_j \in e_1$, then by replacing the path $v_i, e_i, v_{i+1}$ in $C$ with the longer path $v_i, e, u_1, f_1, u_2, \ldots, f_{j-1}, u_j, e_1, v_{i+1}$, we obtain a longer cycle than $C$. Thus $e_1 \cap V(P) = \emptyset$. Then replacing $e_i$ with $e$ in $C$ gives a cycle $C'$ with $(C',P)$ better than $(C,P)$ by criterion (iii).
\end{proof}

Symmetrically, the claim holds for $u_\ell$ as well.

\begin{claim}\label{notneighbor}
For any $u \notin V(C)$, if $u \in e_i$, then $v_i, v_{i+1} \notin N_{H - H_C}(u)$. 
\end{claim}

\begin{proof}
Suppose $v_i \in N_{H - H_C}(u)$, and let $e \in E(H)-E(H_C)$ be such that $\{u, v_i\} \subseteq e$. Then we can find a longer cycle by replacing $e_i$ with $(e, u, e_i)$, a contradiction to our choice of $C$. A similar argument holds for $v_{i+1}$. 
\end{proof}

\begin{claim}\label{distance}For every $e_i \in B_1, e_j \in B_\ell$ either $i=j$ or $|i - j| \geq \ell$.\end{claim}

\begin{proof}
Suppose there exists $e_i \in B_1, e_j \in B_\ell$ such that without loss of generality $j > i$ and $j - i \leq \ell-1$. Then that cycle obtained by replacing $v_i, e_i, \ldots, e_{j}, v_{j+1}$ in $C$ with $v_i, e_i, u_1, f_1, \ldots, f_{\ell-1}, u_\ell, e_j, v_{j+1}$ has size $|V(C)| - (i-j) + \ell > |V(C)|$, a contradiction.
\end{proof}

%
%\begin{lemma}\label{indep} If $C=w_1,\ldots,w_q$ is a graph path, $A$ is any set of $c$  edges of $Q$ and $I$ is an independent subset of
%$\{w_2,\ldots,w_{q-1}\}$ disjoint from all edges in $A$, then $|I|\leq  \lceil \frac{q-1 - c}{2} \rceil$.
%\end{claim}
%
%\begin{proof}
% We show the claim by induction on $q$. Since $c\leq q-1$ and $w_1,w_q\notin I$, the claim holds for $q\leq 2$. 
% For the induction step, let $Q_i$ denote the subpath $w_1,\ldots,w_i$ of $Q$.
% If $w_{q-1}\notin I$, then since at least $c-1$ edges of $A$ are in $Q_{q-1}$,   by induction
% $|I|\leq  \lceil \frac{(q-1)-1 - (c-1)}{2} \rceil=\lceil \frac{q-1 - c}{2} \rceil$, as claimed. So suppose $w_{q-1}\in I$. Then $w_{q-2}\notin I$
% and all $c$ edges of $A$ are in $Q_{q-2}$. Again by induction, $|I\cap Q_{q-2}|\leq \lceil \frac{(q-2)-1 - c}{2} \rceil$, and so
% $|I|\leq 1+|I\cap Q_{q-2}|\leq \lceil \frac{q-1 - c}{2} \rceil$. This proves the claim.
% \end{lemma}

\begin{claim}\label{indep} If $C=v_1,e_1,\ldots,v_s,e_s,v_1$ be a graph cycle, and $A$ is any set of $c$  edges of $C$ and $I$ is an independent subset of
$\{v_1, \ldots, v_s\}$ disjoint from all edges in $A$, then $|I|\leq  \lceil \frac{s-1 - c}{2} \rceil$.
\end{claim}

\begin{proof}
 We show the claim by induction on $s$. If $s = 3$, then either $c \geq 1$, in which case any independent set disjoint from the edges of $A$ has at most one vertex, or $c \geq 2$, and no vertices are disjoint from $A$. Hence we get $|I| \leq \lceil \frac{2 - c}{2} \rceil$.

 Now let $s > 3$ and suppose the lemma holds for $s-1$. If $A = \emptyset$, then $|I| \leq \lfloor s/2 \rfloor = \lceil \frac{s-1}{2} \rceil$, as desired. So suppose $A$ has at least one edge, say $e_i$. Let $C'$ be the cycle obtained by contracting $e_i$. Since $e_i \in A$, $v_i, v_{i+1} \notin I$. Therefore $I$ is still an independent set in $C'$ and is disjoint from the edges in $A - \{e_i\}$. By the induction hypothesis applied to $C', A - \{e_i\},$ and $I$, $|I| \leq \lceil \frac{(s-1) - 1 - (c-1)}{2}\rceil = \lceil \frac{s-1 - c}{2} \rceil$.
 \end{proof}
 
  Claims~\ref{noconsecutive},~\ref{notneighbor} and Lemma~\ref{indep} imply the following corollary.
 
\begin{cor}\label{cycleneighbors}Let $A= \{e_i \in E(C): u_1 \in e_i\}$. Then $|N_{H'}(u_1) \cap V(C)| \leq \lceil \frac{s-1-|A|}{2}\rceil$.
\end{cor}

The following general lemmas will be used in conjunction with Claim~\ref{distance} later in our proof.

\begin{lemma}\label{verc} Let $C=v_1,e_1,\ldots,v_s,e_s,v_1$ be a graph cycle. Let $A$ and $B$ be nonempty subsets of $E(C)$ such that for any $e_i\in A$ and $e_j\in B$ either $i=j$ or $ |i -j| \geq q\geq 2$. 
Suppose $|B|\geq |A|=a$. Then either\\
(a) $a\leq  s/2-q+1$, or
(b) $B=A$ and $a\leq s/q$.
\end{lemma}

\begin{proof} Suppose first $B= A$. Then between any two edges of $A$ on $C$ there are at least $q-1$ other edges. 
This proves (b).

Suppose now $B\neq A$. Let $A=\{e_{i_1},\ldots,e_{i_a}\}$ with vertices in clockwise order on $C$.
 We can view $C$ as the union of  $a$ paths $P_1,\ldots, P_a$ where $P_j$ is the part of $C$ from $e_{i_j}$ to $e_{i_{j+1}}$ (modulo $a$).  
Since $|B|\geq a$, there is some $f\in B-A$, say $f\in P_a$.
Then $P_a$ has at least $2(q-1)$ edges not in $A\cup B$ (and some vertices in $B$). Also, if $e_{i_j}\in A\cap B$, then 
$e_{i_j-1},e_{i_j+1}\notin A\cup B$. This means $|E(C)-A-B|\geq 2(q-1) +(|A\cap B|-1)$ with equality only if $A\subset B$.
Thus if $A\not\subset B$, then since $|B|\geq a$,
\begin{equation}\label{n23}
s\geq |A|+|B-A|+2(q-1) +|A\cap B|\geq 2a+2(q-1),
\end{equation}
as claimed. Otherwise, in view of $f$, $|B|\geq a+1$, and instead of~\eqref{n23}, we get
$$s\geq |A|+|B-A|+2(q-1) +|A\cap B|-1\geq (2a+1)+2(q-1)-1=2a+2(q-1),$$
again.\end{proof}

\begin{lemma}\label{verc2} Let $C=v_1,e_1,\ldots,v_s,e_s,v_1$ be a graph cycle. Let $A$ and $B$ be nonempty independent subsets in $V(C)$ such that for any $v_i\in A$ and $v_j\in B-A$, $|i - j| \geq q\geq 2$. 
If $B-A\neq \emptyset$, then $|A|\leq s/2-q+1$.
\end{lemma}

\begin{proof} Let $A=\{v_{i_1},\ldots,v_{i_a}\}$ with vertices in clockwise order on $C$.
 We  view $C$ as the union of  $a$ paths $P_1,\ldots, P_a$ where $P_j$ is the part of $C$ from $v_{i_j}$ to $v_{i_{j+1}}$ (modulo $a$).  
Since $B-A\neq\emptyset$, we may assume there is $y\in (B-A)\cap E( P_a)$.
Then $P_a$ has at least $2(q-1)$ vertices not in $A\cup B$ and at least one in $B$. 
Since $A$ is independent, we also have at least $a-1$ vertices in $V(C-P_a)-A$.
 Hence $|V(C)|\geq a+a+ 2(q-1)$, as claimed. \end{proof}

\section{Existence of a cycle of length at least $n/2+1$.}

Similarly to Dirac's proof, we show that under the conditions of Theorems~\ref{main3}(c) and~\ref{main41} there exists a cycle of length at least $t+2 \geq n/2+1$. We do this in two cases: $r \leq t$ and $r \geq t+1$. 

\begin{lemma}\label{le1} If $r \leq t$, and $H$ is an $r$-uniform hypergraph with minimum degree $\delta(H) \geq {t \choose r-1} + 1$,  then $H$ contains a cycle of length at least $t+2= \lfloor (n+3)/2 \rfloor.$
%Suppose $3 \leq  r \leq t-1$ and let $H$ be an $r$-graph with $\delta_1(H) \geq {t \choose r-1} +1$. Then there exists a cycle $C$ with $|C|\geq t+2= \lfloor (n+3)/2 \rfloor.$
\end{lemma}
\begin{proof} Suppose $H$ has no cycles of length at least $t+2$.
 Let $Q$ be a longest path in $H$, say $Q = v_1, e_1, v_2, \ldots, e_{s-1} ,v_{s}$. 
 Let $q=\min\{t+1,s\}$,
  $V(q)=\{v_1,\ldots,v_{q}\}$ and let
 $Q(q)$ denote the subpath of $Q$
with vertex set $V(q)$ and edge set $E(q)=\{e_1, \ldots, e_{q-1}\}$. 
 Among such paths $Q$, choose one in which 
 \begin{equation}\label{me}
\parbox{14cm}{\em (a) the most edges in $E(q)$ are contained in $V(q)$, and \\
(b) modulo (a), the fewest edges in $E(q)\cup \{e_{q}\}$ contain $v_1$.
 }
 \end{equation}
 Let ${H_1}=H-E(Q)$. Since $H$ has no cycles of length at least $t+2$ and $Q$ is a longest path, all neighbors of $v_1$ in ${H_1}$ are in $V(q)$. Thus $d_{{H_1}}(v_1)\leq {q-1\choose r-1}$. By the same reason, the edges $e_i$ for $q+1\leq i \leq s-1$ must not contain $v_1$.
 So
 \begin{equation}\label{H1}
  d_H(v_1)\leq d_{{H_1}}(v_1)+\min\{q,s-1\} \leq {q-1 \choose r-1} + \min\{q,s-1\}.
    \end{equation}
 If $q=s \leq t$, then since $3\leq r\leq t-1$, this is at most ${t-1\choose r-1}+t-1\leq {t\choose r-1}$,
 contradicting the minimum degree condition. Hence $s\geq t+1$ and $q=t+1$. Let $E'(q)=E(q)\cup \{e_q\}$ if $e_q$ exists, and $E'(q) = E(q)$ otherwise.
 
 Let $E_0$ be the set of edges in $E'(q)$ not containing $v_1$, $E_1$ be the set of edges in $E'(q)$ containing $v_1$ and  contained in $V(q)$, and $E_2=E'(q)-E_0-E_1$. In particular, $e_q\in E_0\cup E_2$.
 
 Let us show that  \begin{equation}\label{pb} |E_1\cup E_2|\leq \max\{t-1, r\}.
\end{equation}
 
 Indeed, suppose $|E_1\cup E_2|=m$. For every $2\leq i\leq t+1$ such that $v_1 \in e_i$, we can consider the path $Q_i$ from $v_i$ to $v_s$ obtained from $Q$ by replacing the subpath
 $v_1,e_1, v_2,\ldots, e_i, v_{i+1}$ with the subpath $v_i, e_{i-1}, v_{i-1},\ldots,e_1, v_1, e_i, v_{i+1}$. This path uses the same edges as $Q$, so by Rule (a) in \eqref{me} it is also a valid choice for a best path, and if $v_i$ is in fewer than $m$ edges in $E'(q)$, then $Q_i$ is better by Rule~(b). Hence each $v_i$ such that $e_i\in E_1\cup E_2$ is in at least $m$ edges  in $E'(q)$. Thus,
 $m^2\leq r(t+1)$. If $r \leq t-1$, then $m^2\leq t^2 - 1$, so $m \leq t-1$. Otherwise if $r = t$, we get $m \leq t$. This proves \eqref{pb}.
 
 \medskip
 Let $R=R(v_1)$ be the set of $r$-tuples contained in $V(q)$ that contain $v_1$ and are not edges of $H$.
 Since the only edges containing $v_1$ and not contained in $V(q)$ are those in $E_2$,
  \begin{equation}\label{dv1}
d_H(v_1)= {t\choose r-1}+|E_2|-|R|.
 \end{equation}
So, if $E_2=\emptyset$, 
 then  $d_H(v_1)\leq {t\choose r-1}$, a contradiction.
 Hence for some $j\in [t+1]$,  $ e_j\in E_2$, i.e., $x \in e_j$ but $e_j \nsubseteq V(q)$. Choose the smallest such $j$.
 
 \medskip
 {\bf Case 1:} $j=1$. If there is an edge $g\subset V(q)$ in $E(H)-E'(q)$ containing $\{v_1,v_2\}$ (recall that $g \notin \{e_{q+1}, \ldots, e_{s-1}\}$), then by replacing $e_1$ with $g$ we get a contradiction to \eqref{me}(a).  Thus each of the ${t-1\choose r-2}$ $r$-tuples $g\subset V(q)$ containing $\{v_1,v_2\}$ is in $R\cup E_1$.

 \medskip
 {\bf Case 1.1:} $r=3$. For any edge $e_i$ containing $v_1$, $\{v_i, v_{i+1},v_1\}\subseteq e_i$. Then only $e_2$ may contain $\{v_1,v_2\}$ and be contained in $V(q)$. Moreover for $2\leq i\leq t$,  if $x \in e_i$, then $e_i=\{v_1, v_i, v_{i+1}\} \subseteq V(q)$, so $|E_2|\leq 1$.
 Hence
  \[d_H(v_1)\leq {t \choose 2} - |R| + |\{e_2, e_q\}| \leq {t\choose 2}-{t-1\choose 1}+2\leq {t\choose r-1},\] a contradiction.

 \medskip
 {\bf Case 1.2:} $r\geq 4$. Set $E_1' =\{e_i \in E_1: v_2 \in e_i\}$. 
    It follows that
 $$d_H(v_1)\leq |E_2|+{t\choose r-1}-\left({t-1\choose r-2}-|E_1'|\right)=|E_1'\cup E_2| +{t\choose r-1}-{t-1\choose r-2}.
 $$
 In order to have $d_H(v_1)\geq 1+{t\choose r-1}$, we need ${t-1\choose r-2}\leq |E_1'\cup E_2|-1$. 
 
If either $r \leq t-1$ (so $|E_1 \cup E_2|\leq t-1$ by~\eqref{pb}) or $r=t$ and $|E_1' \cup E_2| \leq t-1$, then since $r-2\geq 2$ and $(t-1)-(r-2)\geq 2$, we have ${t-1\choose r-2}\geq {t-1\choose 2}$. We need $\frac{(t-1)(t-2)}{2}\leq t-2$, which does not hold for integer  $t\geq 4$.

Therefore we may assume that $r = t$ and by~\eqref{pb}, $|E_1' \cup E_2| = |E_1 \cup E_2| = t$ by~\eqref{dv1}. Then every $e_i \in E_1$ contains $v_2$. Suppose first that $|E_1| \geq 1$, and let $e_i \in E_1$. If $f:=V(q) - \{v_2\}$ is an edge of $H$, then by assumption $f \notin E(Q)$. We may replace in $Q$ $e_i$ with $f$ and $e_2$ with $e_i$ to obtain a path that is better than $Q$ in criterion (a). It follows that $f \in R$ and 
\[d_H(v_1) \leq {t \choose r-1}  - ({t-1 \choose r-2} +|\{f\}|) + |E_1 \cup E_2| = {t \choose r-1} -(t-1+1) +t = {t \choose r-1},\]
a contradiction. So we may assume that $|E_1| = 0$, i.e., all edges containing $v_1$ in $E'(q)$ contain a vertex outside of $V(q)$. If there exists any edge $e \subseteq V(q)$ in $H$ such that $v_1 \in V(q)$, then  some $\{v_i, v_{i+1}\} \subseteq e$ since $|e| = t$. Then we may replace the edge $e_i$ with the edge $e$ in $Q$ to obtain a better path by criterion (a). Therefore $|R| = {t \choose r-1}$. By~\eqref{dv1}, $d_H(v) \leq |E_2| =t$, a contradiction.

 \medskip
 {\bf Case 2:} $2\leq j\leq t$.  In order for $e_j$ to contain $v_1,v_j,v_{j+1}$ and a vertex outside of $V(q)$, we need $r\geq 4$. Similarly to Case 1, if
 there is an edge $g\subset V(q)$ in $E(H)-E'(q)$  containing $\{v_1,v_{j+1}\}$, then the path 
 $$v_j,e_{j-1}, v_{j-1},\ldots,e_1, v_1, g,v_{j+1}, e_{j+1},v_{j+2},\ldots,e_{s-1}, v_s$$
 contradicts \eqref{me}(a). Hence
each of the ${t-1\choose r-2}$ $r$-tuples $g\subset V(q)$ containing $\{v_1,v_{j+1}\}$ is in $R\cup E_1$.

So, now we repeat the argument of Case 1.2 word by word with $v_j$ in place of $v_2$.

 \medskip
 {\bf Case 3:} $ j= t+1$. This means all edges containing $v_1$ apart from $e_{t+1}$ are contained in $V(q)$. Then $d_H(x) \leq {t \choose r-1}  -|R| + 1$, so we may assume $|R| = 0$. In other words, 
  \begin{equation}\label{t+1}
\mbox{\em 
 all $r$-tuples contained in $V(q)$ and containing $v_1$ are edges of $H$. 
 }
 \end{equation}
 
 Since $r\leq t$, there is $i\leq t$ such that $v_i\notin e_{t+1}$. By \eqref{t+1}, we can construct a path on the vertices
 $v_i, v_{i-1},\ldots, v_1, v_{i+1}, v_{i+2},\ldots, v_{t+1}$ all edges of which are contained in $V(q)$. So, we will have no edges containing $v_i$ and not contained in $V(q)$, a contradiction.
\end{proof}

\bigskip

Next, we prove the result for $r \geq t+1$, assuming that the circumference of $H$ is $k-1$.

\begin{lemma}\label{le2}
If $k \geq r \geq t+1$, $\delta(H) \geq \lceil k/2 \rceil$, and $H$ contains a cycle of length $k-1$, then $H$ contains a cycle of length at least $\min\{k, t+2\}$. 
\end{lemma}

For a path $P= v_1, e_1, v_2, \ldots, e_{\ell-1}, v_\ell$ and $i \in \{1,\ell\}$, let $V_i = V_i(P) = \{v_j \in V(P): v_i \in e_j\}$, and set $V_\ell^+ = V_\ell^+(P) = \{v_{j+1}: v_\ell \in e_j\}$. 

For any $v_i\in V_1$, set \[P_i^1 = v_i, e_{i-1}, \ldots, e_1, v_1, e_i, v_{i+1}, \ldots, e_{\ell-1}, v_\ell,\] and for any $v_j \in V_\ell^+$, set \[P_j^\ell = v_j, e_j, \ldots, e_{\ell-1}, v_\ell, e_j, v_{j-1}, \ldots, e_1, v_1.\]

\begin{lemma}\label{outsideedge}Let $r \geq t+1$, and let $P=v_1, e_1, v_2, \ldots, e_{\ell-1}, v_\ell$ be a longest path in an $r$-uniform hypergraph $H$ with no cycle of length $k$ or greater. Suppose there exists a vertex $v_i \in V_1 \cup V_\ell^+$ and an edge $e \notin E(P)$ such that $v_i \in e$. Then $H$ contains a cycle of length at least $r+1 \geq t+2$.\end{lemma}
\begin{proof}
If there exists an edge $e \notin E(P)$ such that $v_1 \in e$, then by the maximality of $P$, $e \subseteq \{v_1, \ldots, v_\ell\}$. It follows that there exists some $v_q \in e$ with $q \geq r$, and hence $v_1, e_1, \ldots, e_{q-1}, v_q, e$ is a cycle of length $q$. If $q \geq r+1$, then we're done. So we may assume $q =r$ and $e=\{v_1, \ldots, v_r\}$. Swapping $e_1$ with $e$ in $P$, we obtain that since $v_1 \in e_1$, the same argument implies $e_1 =  \{v_1, \ldots, v_r\} = e$, a contradiction.

For $v_i \in V_1$ or $v_j \in V_\ell^+$, we apply the same argument for the longest paths $P_i^1$ or $P_j^{\ell}$ (note $E(P_i^1) = E(P_j^\ell) = E(P)$) and obtain our result.
\end{proof}
%
%
%\begin{lemma}\label{t+2} If $r = t$, then $H$ contains a cycle of length at least $t+2$.
%\end{lemma}
%
%\begin{proof}
%By the previous lemma, $v_1$ is only contained in edge of $E(P)$. Since $\delta(H) \geq t+1$, $v_1$ is contained in an edge $e_q$ with $q \geq t+1$. If $q \geq t+2$, then $(v_1, e_1, \ldots, v_{q}, e_{q}, v_1)$ has length at least $t+2 \geq n/2+1$. Therefore we may assume that $v_1$ is contained only in the edges $e_1, \ldots, e_{t+1}$, and similarly, $v_\ell$ is only contained in the edges $e_\ell, e_{\ell-1}, \ldots, e_{\ell-t}$. 
%
%Since $\delta(H) \geq t+1$, for every $i \leq t$, $v_i$ is contained in an edge outside of $\{e_1, \ldots, e_{t}\}$. Consider the path $P_i^1 = (v_i, e_{i-1}, \ldots, e_1, v_1, e_i, v_{i+1}, \ldots, e_{\ell-1}, v_\ell)$ which is also a longest path. If $v_i$ is contained in an edge $e_q$ with $q \geq t+2$, then again we obtain a cycle of length at least $n/2 + 2$. Therefore $v_i \in e_{t+1}$. It follows that $\{v_1, \ldots, v_t, v_{t+1}\} \subseteq e_{t+1}$, and hence $H$ is not $t$-uniform, a contradiction. 
%\end{proof}
%
%\begin{lemma}\label{n2+1}If $k\geq r \geq t+1 \geq n/2$, $\delta(H) \geq \lfloor (k+1)/2 \rfloor$, and $H$ contains a cycle of length $k-1$, then $H$ contains a cycle of length at least $\min\{k,n/2+1\}$. 
%\end{lemma}
%

Now let $H$ be an-edge maximal counterexample to Lemma~\ref{le2}. That is, $H$ contains a $k-1$ cycle, no cycle of length at least $\min\{k, t+2\}$, and adding any additional edge to $H$ creates such a cycle. 

We break up the proof into 3 parts: when the longest path of $H$ contains $k$ vertices, $k+1$ verties, or at least $k+2$ vertices.

\begin{lemma}\label{notk}The longest path of $H$ contains at least $k+1$ vertices.
\end{lemma}

\begin{proof}
Let $C=v_1,e_1,\ldots,v_{k-1},e_{k-1},v_1$ be a cycle of length at least $k-1$. We may assume that $C$ is a longest cycle and $|C| =k-1 \leq n/2$. Then as $r \geq n/2$ at most one edge of $H$ is contained in $V(C)$ (actually, equals $V(C)$). We may assume that if this happens, then such an edge is one of the $e_i$.

\medskip
{\bf Case 1:} For some $v_1\in V(C)$, some edge $e\in E(H)-E(C)$ contains $v_1$. By our assumption, there is a vertex $u\in e-V(C)$.
Also at most one of $e_1$ and $e_{k-1}$ is contained in $V(C)$, so we may assume there is $u'\in e_1- V(C)$. If $u'=u$ then we have cycle
$C'=v_2,e_2,\ldots,e_{k-1},v_1,e,u,e_1,v_2$ of length $k$, otherwise we have path
$P=u',e_1,v_2,e_2,\ldots,e_{k-1},v_1,e,u$, as claimed.

\medskip
{\bf Case 2:} All edges of $H$ incident to $V(C)$ are in $E(C)$. Since $H$ has at least $k$ edges, there is an edge $f$ fully disjoint from $V(C)$.
Since $|\bigcup_{i=1}^{k-1}e_i|\geq r+1$ and $n<(r+1)+r$,
 there is some $e_i$, say $i=1$ that contains a vertex $u_1\in f$. Let   $u_2$ be another vertex of $f$.  Then we have path 
$P'=v_2,e_2,\ldots,e_{k-1},v_1,e_1,u_1,f,u_2$, as claimed.
\end{proof}

\begin{claim}\label{intersect}
Let $P=v_1, e_1, \ldots, e_{\ell-1}, v_\ell$ be a longest path in $H$ with at least $k+1$ vertices. For every $e \notin E(P)$, $e \cap (V_1 \cup V_{\ell}^+) = \emptyset$.
\end{claim}

\begin{proof}
Suppose there exists $v_i \in V_1 \cup V_{\ell}^+$ such that $v_i \in e$.  If $v_i \in V_1$, the path $P_1^i$ is a longest path but its endpoint $v_i$ is contained in an edge outside of $E(P')$ contradicting Lemma~\ref{outsideedge}. The case for $v_i \in V_{\ell}^+$ is symmetric.
\end{proof}

\begin{lemma}\label{notk+1} The longest path of $H$ contains at least $k+2$ vertices.
\end{lemma}

\begin{proof}Suppose a longest path $P=v_1, e_1, \ldots, e_{\ell-1}, v_{\ell}$ has at most $k+1$ vertices. By Lemma~\ref{notk}, $\ell = k+1$.

If there exists some $v_j \in V_1 \cap V_\ell^+$ (i.e., $v_1 \in e_j$ and $v_\ell \in e_{j-1}$), then the cycle
\[v_1, e_1 \ldots, e_{j-2}, v_{j-1}, e_{j-1}, v_\ell, e_{\ell-1}, \ldots, v_{j+1}, e_j, v_1\] contains all vertices of $P$ except for $v_j$. Therefore $|V(C)|\geq \ell-1 \geq k$, a contradiction. It follows that 
\begin{equation}V_1 \cap V_{\ell}^+ = \emptyset.\end{equation}

Since $|V_1| \geq d_H(v_1), |V_{\ell}^+| \geq d_H(v_\ell)$, we have that $|V_1 \cup V_{\ell}^+| \geq k$, and so at most one vertex in $V(P)$ is not contained in $V_1 \cup V_{\ell}^+$.
\begin{claim}
$|E(H)| = k$. 
\end{claim}
\begin{proof}
Suppose for contradiction that there exists an edge $e \notin E(P)$. By the Claim~\ref{intersect}, $e \cap (V_1 \cup V_{k+1}^+) = \emptyset$. Since $k \geq r \geq n/2$, this is only possible when $k = r = n/2$ and $e$ is the unique edge with $e = V(H) - (V_1 \cup V_{k+1}^+)$. Moreover, this implies that $E(H) = E(P) \cup \{e\}$. So we have $V(H) - V(P) \subseteq e$, and $e$ contains at least $r-1$ vertices outside of $P$. 

If there exists a vertex $v \in e_k - V(P)$, then $v \in e$, and there exists another $v' \in e -V(P) - \{v\}$. We get a longer path by replacing the vertex $v_{k+1}$ with the path $v, e, v'$ in $P$. So $e_k \subseteq V(P)$. Moreover, if there exists a vertex $v_i \in V_1$ such that $v_i \in e_k$, then we obtain the cycle of length $k$ $C' = v_1, e_1, \ldots, v_i, e_k, v_{k}, e_{k-1}, \ldots, v_{i+1}, e_i, v_1$. Hence $V_1 \cap e_k = \emptyset$. Therefore $k+1= |V(P)| \geq |e_k| + |V_i| \geq r + \lceil k/2 \rceil$, but we assumed $r = k \geq 3$, a contradiction. 
\end{proof}

Now suppose we have a cycle of length $k-1$, $C = v_1, e_1, \ldots, e_{k-1}, v_1$. By the previous claim, there exists exactly one edge $e$ such that $e \notin E(C)$.  Among all such pairs $(C, e)$ suppose we chose one to maximize $|e \cap V(C)|$. 

%If $e \subseteq V(C)$, then we must have $k = n/2 + 1$ and $r=n/2$, and $e$ is the unique edge with $e = V(C)$. Then we may swap $e$ with any edge $e_i$ in $C$. The edge $e_i$ contains at least one vertex outside of $V(C)$, contradicting the choice of $(C, e)$. 

Suppose first that $e \subseteq V(C)$. This implies $r = n/2, k = n/2 + 1$. Let $v \in V(H) - V(C)$. We have that $v \notin e$, and $v$ is in at least $k/2$ edges of $C$, hence there exists a consecutive pair of edges, say $e_1, e_2$ containing $v$. The cycle \[C'=v_1, e_1, v, e_2, v_2, e, v_3, e_3, \ldots, v_{k-1}, e_{k-1}, v_1\] has length at least $k$.

Therefore $X:=e \setminus V(C)$ is nonempty. Define $E_X = \{e_i \in E(C): v \in e_i\text{ for some } v\in X\}$. 

\begin{claim}
$E_X$ cannot contain two consecutive edges in $C$.
\end{claim}
\begin{proof}
Suppose $e_1, e_{2} \in E_X$. Then there exists $v, v' \in X$ such that $v \in e_1, v' \in e_{2}$. If $v= v'$, then let $C'$ be the cycle obtained by replacing the vertex $v_{2}$ with $v$. Since $v \in V(C') \cap e$ and we chose $(C,e)$ to maximize $|V(C) \cap e|$, we must have $v_{2} \in e$. We obtain the cycle of length $k$
\[v_1, e_1, v, e, v_2, e_2, v_3, \ldots, v_{k-1}, e_{k-1}, v_1,\] a contradiction.
Therefore we may assume $v\neq v'$. Then by replacing in $C$ the segment $v_1, e_1, v_2, e_2, v_3$ with $v_1, e_1, v, e, v', e_2, v_3$ we obtain a cycle of length $k$.
\end{proof}

So we may assume that since $|E_X| \geq \delta - 1 = \lfloor (k-1)/2 \rfloor$, if $k$ is odd then $E_X = \{e_1, e_3, e_5, \ldots, e_{k-2}\}$ and if $k$ is even, $E_X = \{e_1, e_3, e_5, \ldots, e_{k-3}\}$. Moreover, by the previous claim, $|E_X| = \delta - 1$, and therefore for every $v \in X$, the edges containing $v$ are exactly $E_x \cup \{e\}$. 

Let $e_i \in E_X$, and suppose $v_i \in e$. Then we may replace in $C$ the segment $v_i, e_i, v_{i+1}$ with $v_i, e, v, e_i, v_{i+1}$ for any $v \in X$ to obtain a cycle of length $k$, a contradiction. Similarly, we have $v_{i+1} \notin e$. If $k$ is odd, since $E_x$ contains every other edge of $C$, we have $e \cap C = \emptyset$, i.e., $e = X$. Otherwise, if $k$ is even, then $e \subseteq X \cup \{v_{k-1}\}$.

For every $e_i \in E_X$, $X \cup \{v_i, v_{i+1}\} \subseteq e_i$. When $k$ is odd, we have $|X| = |e| = r$, so $|e_i| \geq r+2$. When $k$ is even, $|X| \geq r-1$, and we get $|e_i| \geq r-1 + 2 \geq r+1$. In either way, we obtain a contradiction.

\end{proof}

We are now ready to prove Lemma~\ref{le2}.
%
%
%
%\section{$\ell\geq k+2$}
%Again, $n/2\leq r\leq (n+1)/2$, $n/2\leq k\leq 1+n/2$ and we have a longest $\ell$-vertex path  $P=v_1,e_1,\ldots,v_{\ell-1},e_{\ell-1},v_\ell$ for
%$\ell\geq k+2$.
%
%\begin{claim}\label{cl1}
%Every edge containing $v_1$ is in $E(P)$.
%\end{claim}
%
%{\bf Proof.} Assume some $e\notin E(P)$ contains $v_1$. If $e\notin V(P)$, we get a longer path. If $e$ contains a $v_i$ with $i\geq r+1$, then we have a cycle of length at least $r+1$.  So, $e=\{v_1,\ldots,v_r\}$. By replacing  in $P$ $e_1$ with $e$, we get a path with the same vertex sequence, but now $e_1$ either has a vertex outside of $P$ or contains $v_i$ for some $i\geq r+1$, a contradiction.\qed
%

%\begin{lemma}\label{l1}
%If $\ell\geq k+2$, then $H$ has a cycle of length at least $\min\{r+1,k\}$.
%\end{lemma} {\bf Proof.} 

{\it Proof of Lemma~\ref{le2}}.
Suppose there does not exist a cycle of length $k$ or greater. Among longest paths choose $P=v_1,e_1,\ldots,v_{\ell-1},e_{\ell-1},v_\ell$ so that $e_1$ has as few vertices outside of $V(P)$ as possible.  By Lemmas~\ref{notk} and~\ref{notk+1}, $\ell \geq k+2$.

Let $J(1)$ be the maximum $j$ such that $v_j\in e_1$ and $J(\ell-1)$ be the minimum $j$ such that $v_j\in e_{\ell-1}$

Then
\begin{equation}\label{o10}
J(1)\leq \min\{r+1,k\}\,\mbox{\em and}\, J(\ell-1)\geq 3.
\end{equation}

Let $i(1)$ be the second smallest index $i$ such that $v_1\in e_i$. By Lemma~\ref{outsideedge}, it is well defined. Similarly, let $i(\ell)$ be the second largest index $i$ such that $v_\ell\in e_i$. Also, let $I(1)$ be the largest index $i$ such that $v_1\in e_i$, and   $I(\ell)$ be the smallest index $i$ such that $v_\ell\in e_i$. 

Since $\ell\geq k+2$ and $H$ has no $k^+$-cycles, 
\begin{equation}\label{o101}
\mbox{\em
$I(1)\leq k-1$, and $I(\ell)\geq 3$.}
\end{equation}

\begin{claim}\label{cl2}
$i(1)\leq i(\ell)$.
\end{claim}

\begin{proof} Suppose $i(1)> i(\ell)$. Then, in view of  the cycle 
$$v_1,e_1,\ldots,e_{i(\ell)-1},v_{i(\ell)},e_{i(\ell)},v_{\ell},e_{\ell-1},v_{\ell-1},\ldots,v_{i(1)+1},e_{i(1)},v_1,$$
$i(1)-i(\ell)\geq 3$. Since $i(\ell)\geq I(\ell)+k/2-2$ and $I(1)\geq i(1)+k/2-2$, by~\eqref{o101},
$$I(1)=(I(1)-i(1))+(i(1)-i(\ell))+(i(\ell)-I(\ell))+I(\ell)\geq (k/2-2)+3+(k/2-2)+3=k+2,$$
 contradicting~\eqref{o101}.
 \end{proof}

\begin{claim}\label{cl3}
$J(1)\leq i(\ell)$.
\end{claim}

\begin{proof} Suppose $J(1)> i(\ell)$. For each $i$ and $j$ such that $v_\ell\in e_i$, $v_j\in e_1$ and $j> i$, the cycle
$C_{i,j}=v_j,e_j,v_{j+1},\ldots,v_\ell,e_i,v_i,e_{i-1},v_{i-1},\ldots,v_2,e_1,v_j$ yields that $j\geq i+3$. 

In particular, by~\eqref{o10}, $i(\ell) \leq k-3$. The edge $e_{i(\ell)}$ forbids $i(\ell)+1$ and $i(\ell)+2$ in $[1,J(1)]$ from be indices of vertices contained in $e_1$. Each of the $d(v_\ell)-2$ edges $e_{i'}$ containing $v_\ell$ with $i'<i(\ell)$ also forbids at least one additional index $v_{i'+1}$ from belonging to $e_1$.
So, by~\eqref{o10}, $|e_1\cap V(P)|\leq k-k/2=k/2$. Hence $|e_1-V(P)|\geq r-k/2$. By the choice of $e_1$, also $|e_{\ell-1}-V(P)|\geq r-k/2$.
Since $(e_1-V(P))\cap (e_{\ell-1}-V(P))=\emptyset$, we conclude that
$$n\geq |V(P)|+|e_1-V(P)|+|e_{\ell-1}-V(P)|\geq \ell+(r-k/2)+(r-k/2)\geq (k+2)+2r-k=2r+2,$$
a contradiction.\end{proof}

\medskip
Define $i'(1)=\max\{2,i(1)-1\}$ and $i'(\ell)=\min\{\ell-2,i(\ell)+1\}$. If $i(\ell)\leq \ell-3$, then let 
$$P'(\ell)=v_1,e_1,\ldots, v_{i(\ell)},e_{i(\ell)},v_\ell, e_{\ell-1},v_{\ell-1},\ldots,e_{i(\ell)+1},v_{i(\ell)+1}.$$ If $i(\ell)=\ell-2$ and $v_{\ell-2}\in 
e_{\ell-1}$,   then let 
$P'(\ell)=v_1,e_1,\ldots, v_{\ell-2},e_{\ell-1},v_{\ell-1} ,e_{\ell-2},v_{\ell}$. In both cases,
\begin{equation}\label{o11}
\parbox{14cm}{\em
$P'(\ell)$ coincides with $P$ up to $v_{i(\ell)}$, has the same vertex set as $P$, and the last edge is $e_{i'(\ell)}$.}
\end{equation}

As we mentioned above, $(e_1-V(P))\cap (e_{\ell-1}-V(P))=\emptyset$. Also, if $v_j\in e_{\ell-1}$ and $v_{j+1}\in e_1$, then the cycle
$v_2,e_2,v_3,\ldots,v_j,e_{\ell-1},v_{\ell-1},e_{\ell-2},\ldots,v_{j+1},e_1,v_2$ has at least $k$ vertices. Thus, $e_1$ ``forbids" for $e_{\ell-1}$
 $|e_{\ell-1}-V(P)|$ vertices outside of $V(P)$ and $|e\cap V(P)|-1$ vertices in $V(P)$. By~\eqref{o101}, $v_1$ and $v_2$ also cannot belong to
 $e_{\ell-1}$.
 
Now we consider some cases. Suppose first that $v_3\notin e_1$. Then $r$ vertices are forbidden for $e_{\ell-1}$. And by~\eqref{o101}, $v_{k-1},v_k,\ldots,v_\ell$ are not forbidden, so $e_{\ell-1}$ contains them all, in particular, $v_{\ell-2}\in 
e_{\ell-1}$. By Claim~\ref{cl3}, all $r$ forbidden vertices for $e_{\ell-1}$ are outside of $P$ or in the set $\{v_1,\ldots,v_{J(1)-1}\}$. Thus, by~\eqref{o11}, they are also forbidden for $e_{i'(\ell)}$. But $|e_{i'(\ell)}\cup e_\ell| \geq r+1$, a contradiction.

So suppose $v_3\in e_1$. If $v_1\in e_2$, then let $P_1=v_1,e_2,v_2,e_1,v_3,e_3,\ldots,v_\ell$. Since by~\eqref{o10}, $J(\ell-1)\geq 3$, each vertex of $e_2-e_1$ forbids an extra vertex for $e_{\ell-1}$. So, again $r$ vertices are forbidden for $e_{\ell-1}$. Thus repeating the argument as the previous paragraph again yields a contradiction.
Therefore we may assume $v_1 \notin e_2$. If there is a vertex $v\in e_2-V(P)$, then let $P_2=v,e_2,v_2,e_1,v_3,e_3,\ldots,v_\ell$. This path differs from $P$ only in the first vertex, so as before, each vertex of $e_2-e_1$ forbids an extra vertex for $e_{\ell-1}$. Thus repeating the argument of the previous paragraph again yields a contradiction. If $e_2$ contains a vertex $v_i$ for some $i\geq r+2$. Then the cycle $v_3,e_3,v_4,\ldots,v_i,e_2,v_2,e_1,v_3$ has $i-1\geq r+1$ vertices. 

The remaining case is $e_2=\{v_2,v_3,\ldots,v_{r+1}\}$. If for some $3\leq i\leq r$, $v_i\in e_{\ell-1}$, then the cycle
$C_i=v_{i+1},e_{i+1},\ldots,v_{\ell-1},e_{\ell-1},v_i,e_{i-1},v_{i-1},\ldots,v_3,e_1,v_2,e_2,v_{i+1}$ has $\ell-2\geq k$ vertices. Thus
vertices $v_1,\ldots,v_r$ are forbidden for $e_{\ell-1}$. It follows that $e_{\ell-1}=V(H)-\{v_1,\ldots,v_r\}$. Thus $P'(\ell)$ exists.
If $3\leq i(\ell)\leq r$, then  the cycle
$v_{i(\ell)+1},e_{i(\ell)+1},\ldots,v_{\ell},e_{i(\ell)},v_{i(\ell)},e_{i(\ell)-1},v_{i(\ell)-1},\ldots,v_3,e_1,v_2,e_2,v_{i(\ell)+1}$ has $\ell-1\geq k+1$ vertices.
Thus $i(\ell)\geq r+1$, and hence the last edge $e_{i'(\ell)}$ of $P'(\ell)$ also is disjoint from $\{v_1,\ldots,v_r\}$. This is a contradiction. \qed

\section{The path $P$ in a  best pair $(C,P)$ is nontrivial}\label{secl1}

Consider a best pair $(C,P)$ with $C = v_1, e_1, v_2, \ldots, e_{s-1}, v_s, e_s, v_1$ and $P = u_1, f_1, u_2, \ldots, f_{\ell-1}, u_\ell$. In this section, we rule out the case that $P$ contains only one vertex, i.e., $\ell=1$.

%We consider three subgraphs, $H_C,H_P$ and $H'$ of $H$ with the same vertex set $V(H)$:  $E(H_C)=\{e_1,\ldots,e_s\}$, $E(H_P)=\{f_1,\ldots,f_{\ell-1}\}$ and  $E(H')=E(H)-E(H_C)-E(H_P)$.

Observe that if %$u_1$ is the unique vertex in $P$ 
$\ell=1$ 
and $(C,P)$ is a best pair, then every edge of $H'$ contains at most one vertex outside of $V(C)$, otherwise we find a longer  path. 
%The following two claim holds regardless of the values of $r$ and $k$.

%
%\begin{claim}\label{notneighbor}
%If $u_1 \in e_i$, then $v_i, v_{i+1} \notin N_{H'}(u_1)$. 
%\end{claim}
%
%\begin{proof}
%Suppose $v_i \in N_{H'}(u_1)$, and let $e \in E(H')$ be such that $u_1, v_i \in e$. Then we can find a longer cycle by replacing $e_i$ with $(e, u_1, e_i)$, a contradiction to our choice of $C$. A similar argument holds for $v_{i+1}$. 
%\end{proof}
 
\subsection{The case of $\ell=1$ and $r>t$}

 In this section we prove the following lemma.

\begin{lemma}\label{ell1}
Let $n$, $k$, and $r$ be positive integers such that $n \geq k$ and $r > t$. If $H$ is an $r$-uniform hypergraph with at least $k$ edges such that
$\delta(H) \geq \lceil \frac{k}{2} \rceil$ and $c(H) < k$, then $\ell = |V(P)| \geq 2$.
\end{lemma}
%Since $r>t$, we have $\delta(H) \geq \lceil k/2 \rceil$. 

\begin{proof}Suppose $\ell = 1$. We consider two cases. % based on the edges containing $u_1$.

\medskip
{\bf Case 1:} $u_1$ is contained in some $e \in E(H')$.  By Claim~\ref{noconsecutive}, no two vertices of $e$ can be consecutive on $C$. Since $e$ contains $r-1$ vertices of $C$, this gives $r-1 \leq \lfloor s/2 \rfloor$. We know that $r \geq n/2$, so this implies that either $s = n-2$ and $n$ is even, or $s = n-1$. In either case, there are at most two vertices in $V(C)-e$ that are consecutive along $C$. Thus any edge $f \in E(H')$ with $f \neq e$ containing $u_1$ must have the property that $v_i \in e$ and $v_{i+1} \in f$ for some $i$. However, replacing $e_i$ in $C$ with $e, u_1, f$ extends $C$, so such an edge $f$ cannot exist. If $u_1 \in e_j$, then $v_j, v_{j+1} \notin e$ by Claim~\ref{notneighbor}. Thus $u_1$ is contained in at most one edge in $E(H')$ and at most one edge in $E(C)$. So $\lceil k/2 \rceil \leq \delta(H) \leq d_H(u_1) \leq 2$, which can only be true if $k \in \{3,4\}$. Since $3 \leq t+2 \leq s \leq k-1 \leq 3$, $s = 3$, and therefore $e$ must contain at least 2 consecutive vertices in $C$, contradicting Claim~\ref{noconsecutive}. 

\medskip
{\bf Case 2:} $u_1$ is only contained in edges of $C$.

\medskip
{\bf Case 2.1:} There is some edge $e \in E(H')$ with $e \subseteq V(C)$. If $v_i, v_j \in e$ and $u_1 \in e_i, e_j$ for some $i < j$, then the cycle \[ v_1, e_1, v_2, \ldots, e_{i-1}, v_i, e, v_j, e_{j-1}, v_{j-1}, \ldots, e_{i+1}, v_{i+1}, e_i, u_1, e_j, v_{j+1}, e_{j+1}, \ldots, e_{s-1}, v_s, e_s, v_1
\]

is longer than $C$, a contradiction. A symmetric longer cycle can be found if $u_1 \in e_{i-1}, e_{j-1}$. Thus $u_1$ is contained in at most one edge of $\{e_i: v_i \in e\}$ and at most one edge of $\{e_{i-1}: v_i\in e\}$.

If the vertices of $e$ are not all consecutive along $C$, then there are at least $r+2$ edges in $\{e_i: v_i \in e\} \cup \{e_{i-1}: v_i \in e\}$. Since $u_1$ is contained in at most two such edges,  $e$ prohibits at least $r$ edges of $C$ from containing $u_1$. Since $u_1$ is contained in at least $k/2$ edges of $C$, we have \[r+k/2 \leq s \leq k-1,\] which implies $r \leq k/2-1$, contradicting that $r \geq n/2 \geq k/2$. 

If the vertices of $e$ are consecutive along $C$, by symmetry say $e = \{v_1, \ldots, v_r\}$, then $e$ prohibits at least $r-1$ edges of $C$ from containing $u_1$, so \[r-1+k/2 \leq s \leq k-1.\] This implies $r \leq k/2$, which gives a contradiction unless $k=n$ is even, $r=n/2$, $s=k-1 = n-1$, and $u_1$ is contained in exactly two edges of $\{e_i: v_i \in e\} \cup \{e_{i-1}: v_i \in e\}$. The last condition implies that $u_1$ must be contained in $e_r$ and $e_s$ because any other such edge $e_i$ satisfies that $v_i, v_{i+1} \in e$. 
%Choose an $i$ such that $v_j, v_{j+1} \in e$ and $u_1 \notin e_i$, and consider the cycle $C'$ formed by replacing $e_i$ with $e$ in $C$. Since $s = n-1$ and $u_1 \notin V(C)=V(C')$, we have that $e_i \subseteq V(C')$. We must have that $e_i \neq e$, so by the same arguments above, $e_i$ must prohibit at least one additional edge of $C$ from containing $u_1$. Then at least $r=k/2$ edges of $C$ cannot contain $u_1$ and $k/2$ edges of $C$ must contain $u_1$, contradicting that $s = k-1$.
Now consider the cycle $C'$ formed by replacing $e_{r-1}$ with $e$ in $C$. Since $s = n-1$ and $u_1 \notin V(C)=V(C')$, we have that $e_{r-1} \subseteq V(C')$. Let $v_i \in e_{r-1} - e$ (so $i \in \{r+1, \ldots, s\}$). Since $u_1 \in e_r$ and $v_r \in e_{r-1}$, the same argument applied to $C'$ and $e_{r-1}$ implies that $u_1 \notin e_i$.  Thus $e_{r-1}$ prohibits $u_1$ from belonging to an additional edge of $C$.  It follows that at least $r=k/2$ edges of $C$ cannot contain $u_1$ and $k/2$ edges of $C$ must contain $u_1$, contradicting that $s = k-1$.

\medskip
{\bf Case 2.2:} Each $e \in E(H')$ contains exactly one vertex $v \notin V(C)$. Since $C$ has at most $k-1$ edges, and $|E(H)| \geq k$, $E(H') \neq \emptyset$. Fix an edge $e \in E(H')$ and corresponding vertex $v \notin V(C)$. We must have $v \neq u_1$ because $u_1$ is contained only in edges of $C$. As before, $u_1$ is contained in at most one edge from each set $\{e_i: v_i \in e\}$ and $\{e_{i-1}: v_i \in e\}$. If the vertices of $e \cap V(C)$ are not all consecutive along $C$, then $e$ prohibits at least $r-1$ edges of $C$ from containing $u_1$. Since $u_1$ must be contained in at least $k/2$ edges of $C$, we have \begin{equation}\label{rbound} r-1+k/2 \leq s \leq k-1,\end{equation} which implies $r \leq k/2$. This gives a contradiction unless $k=n$ is even, $r=n/2$, and $s=k-1$. However, $u_1$ and $v$ are both outside of $C$, so $s \leq n-2 = k-2$, a contradiction.

If the vertices of $e \cap V(C)$ are consecutive along $C$, then $e$ prohibits at least $r-2$ edges of $C$ from containing $u_1$, so \[r-2+k/2 \leq s \leq \min\{k-1,n-2\}.\] This implies $r \leq k/2 +1$, which gives a contradiction if $k \leq n-3$. 

If $k \geq n-2$, then we get a contradiction unless $s = \min\{k-1,n-2\}$ and $r = \lceil n/2 \rceil$. If there exists some $f \in E(H')$ with $v \in f$ and $f \neq e$, then $f$ prohibits at least one additional edge of $C$ from containing $u_1$, using the same arguments as for $e$. In this case, we have $r-1+k/2 \leq s$, which gives a contradiction similar to \eqref{rbound}. Otherwise, $v$ must be contained in at least $k/2-1$ edges of $C$. If $v_i \in e$ then $v \notin e_i, e_{i-1}$ by Claim~\ref{notneighbor}. Thus $e$ prohibits at least $r$ edges of $C$ from containing $v$, so $r+k/2-1 \leq s$, giving the same contradiction as \eqref{rbound}.
\end{proof}

\subsection{The case of $\ell=1$ and $r=t$}

We first prove a claim that will be used in this section and the following.

\begin{claim}\label{b12}
Let $n$, $k$, and $r$ be positive integers such that $n \geq k$ and $r \leq t$. If $H$ is an $r$-uniform hypergraph with at least $k$ edges such that
$\delta(H) \geq {t \choose r-1}+1$, $c(H) < k$, and $\ell = 1$, then $u_1$ is contained in at least 2 edges of $C$.
\end{claim}
\begin{proof}
Suppose that $u_1$ is contained in at most one edge of $C$. 
By Claim~\ref{noconsecutive} no two vertices of $N_{H'}(u_1)$ are consecutive. Since $s \leq n-1 \leq 2t+1$, this implies that $|N_{H'}(u_1)\cap V(C)|\leq t$. But since $\ell=1$, $N_{H'}(u_1)\subset V(C)$. So, since $|N_H(u_1)|\geq t+1$,
$u_1$ must be contained in an edge of $C$, say $u_1 \in e_{2t}$. Then by Claims~\ref{notneighbor}
and~\ref{noconsecutive},
 the ${t\choose r-1}$ edges of $H'$ containing $u_1$ must be disjoint from $\{v_{2t},v_{2t+1}\}$ and nonconsecutive along $C$. This is possible only if $s=2t+1$ and $|N_{H'}(u_1)\cap V(C)|= t$.

 We may assume that $X:= N_{H'}(u_1) = \{v_1, v_3, \dots, v_{2t-1}\}$. Then $u_1$ must be contained in the ${t \choose r-1}$ edges of $H'$ consisting of $u_1$ and $r-1$ vertices of $X$.

We now will find an edge $g \neq e_{2t}$ such that $|g-X|\geq 2$ and $|g \cap \{v_2, v_4, \dots, v_{2t-2}\}| \geq 1$. To do so, choose $v_{2j} \notin e_{2t}$. Since $d_H(v_{2j})>{t\choose r-1}$, there is an edge $g$ containing $v_{2j}$ and at least one additional vertex not in $X$. Notice that this vertex cannot be $u_1$, so it must be either $v_{2t+1}$ or be $v_{2j'}$ for some $1 \leq j' \leq t$, $j' \neq j$. 

We use $g$ to find a hamilitonian cycle. First suppose that $g \in E(H')$. Let $f_{2j-1}$ be an edge in $E(H')$ containing both $u_1$ and $v_{2j-1}$, which must exist because $v_{2j-1} \in X$. If $v_{2t+1} \in g-X$, then we obtain the hamiltonian cycle \[C_1 = v_{2j}, g, v_{2t+1}, e_{2t+1}, v_1, e_1,  \dots, v_{2j-1}, f_{2j-1}, u_1, e_{2t}, v_{2t}, e_{2t-1}, \dots, v_{2j}. \] Otherwise, we have $v_{2j'} \in g-X$ for some $1 \leq j' \leq t$, $j'\neq j$. Let $f_{2j'-1} \neq f_{2j-1}$ be an edge of $H'$ containing both $u_1$ and $v_{2j'-1}$. Then the cycle \[v_{2j},g,v_{2j'},e_{2j'}, v_{2j'+1},e_{2j'+1}, \dots ,v_{2j-1},f_{2j-1},u_1,f_{2j'-1},v_{2j'-1},e_{2j'-2}, \dots ,v_{2j} \] is hamiltonian.

Now we may assume that $g = e_i$ for some $i \neq 2t$. By symmetry, we may assume $i$ is odd. Let $f_i \neq f_{2j-1}$ be an edge of $H'$ containing both $u_1$ and $v_i$. If $2j \neq i+1$, then we have the hamiltonian cycle \[C_2 = v_{2j},g,v_{i+1},e_{i+1}, v_{i+2},e_{i+2}, \dots ,v_{2j-1},f_{2j-1},u_1,f_{i},v_{i},e_{i-1}, \dots ,v_{2j}.\] If $2j=i+1$ and $v_{2t+1} \in g-X$, then $g = e_{2j-1}$ and the we obtain the cycle $C_1$. Otherwise, $2j= i+1$ and there is some $v_{2j'} \in g-X$ with $j \neq j'$. Swapping the role of $j'$ with $j$ in the cycle $C_2$ gives a hamiltonian cycle.

\end{proof}

\begin{lemma}\label{ell2}
Let $n$, $k$, and $r$ be positive integers such that $n \geq k$ and $r = t$. If $H$ is an $r$-uniform hypergraph with at least $k$ edges such that
$\delta(H) \geq r+1$ and $c(H) < k$, then $\ell = |V(P)| \geq 2$.
\end{lemma}

%Since $r=t$, we have $\delta(H) \geq r+1$. This gives that $H$ has at least $\frac{n(r+1)}{r} > n+2$ edges.

\begin{proof}Suppose $\ell = 1$. We consider cases based on the edges containing $u_1$ and the edges outside of $C$. 

\medskip
{\bf Case 1:} $u_1$ is contained in some $e \in E(H')$. Note that no two vertices of $e \cap V(C)$ can be consecutive by Claim~\ref{noconsecutive}. Thus $r-1 \leq \lfloor s/2 \rfloor$, so $s \geq n-3$. Observe also that by Claim~\ref{notneighbor}, if $v_i \in e$, then $u_1 \notin e_i, e_{i-1}$. Thus we have $n-3 \leq s \leq n-1$, and there are at most three edges $e_i$ in $C$ with $v_i, v_{i+1} \notin e$. 

\medskip
{\bf Case 1.1:} There are at most two $e_i$ in $C$ with $v_i, v_{i+1} \notin e$. Then there are at least $r+1 - 2 \geq 2$ edges of $E(H')$ containing $u_1$, so consider $f \in E(H')$ with $u_1 \in f \neq e$. If for some $i$, $v_i \in e$ and $v_{i+1} \in f$ (or vice versa), we replace $e_i$ with $e,u_1, f$ to obtain a longer cycle. If no such $i$ exists, then for all $v_j \in f$ we have that $v_{j-1}, v_{j+1} \notin e$. Since $f \neq e$, we can fix a $j$ such that $v_j \in f-e$. Then $f$ prohibits $e_{j-1}$ and $e_j$ from containing $u_1$, which were not prohibited by $e$. Therefore no edges of $C$ contain $u_1$, so there are at least $r+1$ edges in $E(H')$ containing $u_1$. Then there must exist some such $f' \in E(H')$ and some $i$ such that $v_i \in f'$ and $v_{i+1}$ is in $e$ or $f$, which allows us to replace $e_i$ and obtain a longer cycle.

\medskip
{\bf Case 1.2:} There are three edges $e_i$ in $C$ with $v_i, v_{i+1} \notin e$. This case can only occur when $s=n-1$ and $n$ is even, so we have $s = 2t+1$. We first suppose that $r \geq 4$ and deal with the case $r=3$ separately. Thus we have at least $r+1-3 \geq 2$ edges of $E(H')$ containing $u_1$. As in Case 1.1, we consider $f \in E(H')$ with $u_1 \in f \neq e$, and we may assume that for all $v_j \in f$ we have that $v_{j-1}, v_{j+1} \notin e$. We also have some $j$ such that $v_j \in f-e$, which gives that $u_1 \notin e_{j-1}, e_j$. Thus at most one edge of $C$ contains $u_1$. 

If there is more than one vertex in $f'-e$ for any $f' \in E(H')$ containing $u_1$, then no edges of $C$ contain $u_1$ and we can repeat the arguments of Case 1.1 to obtain a longer cycle. By symmetry, the same holds for the edge $f$, so $N_{H'}(u_1) = e \cup f$. Notice that $|e \cup f| = r$, so there are at most $r$ edges of $E(H')$ containing $u_1$. Since $d(u_1) \geq r+1$, this gives that $u_1$ is contained in exactly those $r$ edges along with one edge of $C$, contradicting Claim~\ref{b12}.

We now handle the case $r=3$. Notice that in this case, $n=8$ and $s=7$. If $u_1$ is contained in at least two edges of $H'$, then we can in fact follow the above arguments. Thus we may assume that $u_1$ is contained in exactly one edge of $H'$ and three edges of $C$. Up to symmetry, we have two cases.

First, consider the case $u_1 \in e = \{u_1, v_2, v_5\}$ and $u_1 \in e_3, e_6, e_7$. The cycle $C_1 = v_1, e_1, v_2, \dots, v_6, e_6, u_1, e_7, v_1$ has the same edge set as $C$ and misses only the vertex $v_7$. If $v_7$ is not contained in an $H'$ edge, then $(C_1, v_7)$ is a better pair than $(C, u_1)$, a contradiction. Then $v_7 \in f \in E(H')$, and observe that $f$ cannot contain any vertex in $\{u_1, v_1, v_6\}$ by Claim~\ref{notneighbor} since $v_7 \in e_6, e_7$.

We now consider the possibilities for the edge $f$. If $v_3 \in f$, then we obtain the hamilitonian cycle $v_7, f, v_3, e_3, \dots, v_6, e_6, u_1, e, v_2, e_1, v_1, e_7, v_7$. A symmetric argument gives a hamiltonian cycle when $v_4 \in f$. Thus $f = \{v_7, v_2, v_5\}$, and $f$ must be the only $H'$ edge containing $v_7$. Then $v_7 \in f, e_6, e_7$, and some $e' \in E(C)$. By Claim~\ref{notneighbor}, $e' \neq e_1,e_2, e_4, e_5$. Thus $e' = e_3$, but we already have $e_3 = \{u_1, v_3, v_4\}$.

The second case for $r=3$, up to symmetry, has $u_1 \in e = \{u_1, v_2, v_4\}$ and $u_1 \in e_5, e_6, e_7$. We consider the same cycle $C_1$ as above, and again we have the edge $f \in E(H')$ containing $v_7$ such that $f$ cannot contain any vertices in $\{u_1,v_1, v_6\}$. 

If $v_3 \in f$, we obtain the hamiltonian cycle $v_7, f, v_3, e_2, \dots, v_1, e_7, u_1, e, v_4, e_4, \dots, e_6, v_7$. If $v_5 \in f$, we have the hamiltonian cycle $v_7, f, v_5, e_5, v_6, e_6, u_1, e, v_4, e_3, \dots, v_1, e_7, v_7$. Thus $f = \{v_7, v_2, v_4\}$, and $f$ must be the only $H'$ edge containing $v_7$. By Claim~\ref{notneighbor}, $e' \neq e_1,e_2, e_3, e_4$, so $e' = e_5$. But we already have $e_5 = \{u_1, e_5, e_6\}$.

\medskip
{\bf Case 2:} $u_1$ is only contained in edges of $C$.

\medskip
{\bf Case 2.1:} There is some edge $e \in E(H')$ with $e \subseteq V(C)$. 

As in the case $r >t$, $e$ prohibits many edges of $C$ from containing $u_1$. That is, $u_1$ is contained in at most one edge of $\{e_i: v_i \in e\}$ and at most one edge of $\{e_{i-1}: v_i\in e\}$. 

If the vertices of $e$ are not all consecutive along $C$, then $e$ prohibits at least $r$ edges from containing $u_1$. Since $u_1$ must be contained in at least $r+1$ edges of $C$, we have \[r+r+1 \leq s.\] Thus we know $r \leq (s-1)/2$, so we reach a contradiction unless $k=n$ and $s=n-1$. Notice that if the vertices of $e$ are in more than two consecutive strings in $C$, then $e$ prohibits at least $r+1$ edges and we reach a contradiction. Assume without loss of generality that $e = \{v_1, \dots, v_{i_1}, v_{i_2}, \dots, v_{i_3}\}$ with $i_2 \geq i_1+2$ and $i_3 \leq s-1$. We must also have that $u_1$ is contained in each edge $e_i$ of $C$ such that $v_i, v_{i+1} \notin e$, and $u_1$ is contained in exactly one of $e_{i_1}, e_{i_3}$ and exactly one of $e_{i_2-1}, e_{n-1}$.

Suppose first that $u_1$ is contained in $e_{i_1}$ and $e_{i_2-1}$. Let $f \in E(H')$ with $f \neq e$. Since $u_1$ is the only vertex outside of $C$, $f \subseteq V(C)$. If there is some $v_i \in f$ such that $u_1 \notin e_{i-1}, e_i$, then $f$ prohibits at least one additional edge from containing $u_1$, giving a contradiction. Thus $f \subseteq e \cup \{v_{i_3+1}, v_{n-1}\}$. Since $f \neq e$, $f$ must contain at least one of $v_{i_3+1}, v_{n-1}$. However, if $v_{i_3+1} \in f$, then $v_{i_1} \notin f$ because $u_1 \in e_{i_3+1}, e_{i_1}$, and similarly if $v_{n-1} \in f$, then $v_{i_2} \notin f$. Therefore we have three distinct possibilities for $f$ ($f = e-v_{i_1}+v_{i_3+1}$, $f= e-v_{i_2}+v_{n-1}$, and $f = e -v_{i_1}-v_{i_2}+v_{i_3+1}+v_{n-1}$), and there are at least $n+3 - (n-1) - 1 = 3$ edges in $E(H')$ distinct from $e$. Hence each of the three possibilities are edges in $H'$.  Notice also that for any $e_i$ such that $v_i, v_{i+1} \in e$, we can swap $e$ and $e_i$ to get another maximum cycle. Since $e_i \neq e$ and $e_i \neq f$, $f \in E(H')$, we must have that $e_i$ forbids at least one additional edge from containing $u_1$, a contradiction.

Now suppose instead that $u_1$ is contained in $e_{i_1}$ and $e_{n-1}$. Let $f \in E(H')$ with $f \neq e$. As in the paragraph above, we have $f \subseteq e \cup \{v_{i_2-1}, v_{i_3+1}\}$, unless $i_1 = 1$, which we will handle separately. If $i_1 \neq 1$, then by a similar argument to above we reach a contradiction. If $i_1 = 1$, notice that $u_1$ must be contained in $r+1$ consecutive edges of $C$: $e_{i_3+1}, e_{i_3+2}, \dots, e_{n-1}, e_1, e_2, \dots, e_{i_2-2}$. In this case, either $f \subseteq (e - \{v_1\}) \cup \{v_{i_2-1}, v_{i_3+1}\}$ or $f = e -v_1 + v_i$ for some $v_i \notin e$. Similarly, for any $e_j$ such that $v_j, v_{j+1} \in e$, we must have $e_j \subseteq (e - \{v_1\}) \cup \{v_{i_2-1}, v_{i_3+1}\}$ or $e_j = e -v_1 + v_i$, $v_i \notin e$, because otherwise we may swap $e$ for $e_i$ to see that an additional edge of $C$ is prohibited from containing $u_1$. This gives that no $f \in E(H')$, $f \neq e$ and no $e_j$, $i_2 \leq j \leq i_3-1$ contains $v_1$. 

Consider the cycle $C'$ formed by swapping $u_1$ with $v_1$ and $e$ with the center edge amongst $e_{i_2}, e_{i_2+1}, \dots, e_{i_3-1}$, call it $e_k$. That is \[C' = u_1, e_1, v_2, e_2, v_3, \dots , e_{k-1}, v_k, e, v_{k+1}, e_{k+1}, v_{k+2}, \dots, e_{n-1}, u_1.\] Then $v_1$ is contained only in edges of $C'$, so $C'$ is an optimal choice of cycle under the same conditions as $C$. If the edges of $C'$ containing $v_1$ are not all consecutive in along $C'$, then we must be done by a previous argument applied to $C'$ instead of $C$. If $r \geq 5$, then we immediately see that $v_1 \in e$ but $v_1 \notin e_{k-1}, e_{k+1}$, so we are done. If $r = 3,4$, then we may assume $k = i_2$ and say $v_1 \in e_{i_2-1}, e_{i_2-2}, e_{i_2-3}$ in order for the edges of $C'$ containing $v_1$ to be consecutive. Then any $f \in E(H')$ with $v_{i_2} \in f \neq e$ must have $v_{i_2-1}, v_{i_2+1} \notin f$, since if $v_i, v_j \in f$, then $v_1$ cannot be in both $e_i, e_j$ and cannot be in both $e_{i-1}, e_{j-1}$. However, there is no such edge $f \in E(H')$, so no such $f$ contains $v_{i_2}$. There is exactly one possibility for $f$ not containing $v_{i_2}$: $f = (e - \{v_1, v_{i_2}\}) \cup \{v_{i_2-1}, v_{i_3+1}\}$. This contradicts that we have at least 3 edges in $E(H')$ distinct from $e$.

We may now assume that all edges of $E(H')$ contained entirely in $V(C)$ are each consecutive in $C$, and that $e = \{v_1, v_2, \ldots, v_r\}$. Then $e$ prohibits at least $r-1$ edges of $C$ from containing $u_1$, so \[r-1+r+1 \leq s\] and thus $r \leq s/2$. If $s \leq n-3$, we immediately get a contradiction. If $s = n-2$, there exists a unique $v \notin V(C)$ with $v \neq u_1$. We must have $u_1 \in e_r, e_{r+1}, \dots, e_n$ because otherwise $e$ prohibits $r$ edges of $C$ from containing $u_1$ and we reach a contradiction. Further, we must have that each edge in $E(H')$ contains $v$, since any additional consecutive edge of $H'$ contained entirely in $V(C)$ would prohibit at least one additional edge from containing $u_1$. Thus $v$ is contained in at least 4 edges of $E(H')$. 

For $e_v \in E(H')$ containing $v$, we have that if $v_i, v_j \in e_v \cap V(C)$, then $u_1$ cannot be contained in both $e_i$ and $e_j$ and cannot be contained in both $e_{i-1}$ and $e_{j-1}$. Thus we must have that any such $e_v$ can contain at most one vertex outside $e \cup \{v\}$, and further that if $e_v$ contains some vertex outside of $e \cup \{v\}$, then $v_1, v_r \notin e_v$. Therefore there exist $e_v, e'_v$ containing $v$ and $v_i, v_{i+1} \in V(C)$ such that say $v_i \in e_v$ and $v_{i+1} \in e_v'$. We are able to extend the cycle $C$ by replacing $e_i$ with $e_v, v, e_v'$, contradicting the maximality of $C$. 

Therefore we must have $s = n-1$. Then $u_1$ is the only vertex outside of $C$, so there are at least 4 edges of $E(H')$, including $e$, each with their vertices consecutive along $C$. This prohibits at least $r+1$ edges of $C$ from containing $u_1$, giving a contradiction.

\medskip
{\bf Case 2.2:} Each $e \in E(H')$ contains some $v \notin V(C)$.

Let $e$ be such an edge and $v \neq u_1$ the unique vertex in $e - V(C)$. Note that as in the previous case, $u_1$ is contained in at most one edge of $\{e_i: v_i \in e \cap V(C)\}$ and at most one edge of $\{e_{i-1}: v_i\in e\cap V(C)\}$.

If the vertices of $e \cap V(C)$ are not all consecutive along $C$, then $e$ prohibits at least $r-1$ edges of $C$ from containing $u_1$. Thus \[r-1+r+1 \leq s,\] so $r \leq s/2$. If $s \leq n-3$, we immediately get a contradiction. Since $u_1, v \notin V(C)$, we must have $s = n-2$ and thus every edge of $H'$ contains $v$. Hence $v$ is contained in at least $(n+3)-(n-2) = 5$ edges of $E(H')$. For $e, f \in E(H')$, if $v_i \in e$, $v_{i+1} \in f$ for some $i$, then we can replace $e_i$ with $e,v,f$ to extend $C$. Since $e$ is not all consecutive, it prohibits at least $r+2$ vertices of $C$ from being contained in $f$. However, $C$ has at most $2r$ vertices and $f$ must contain at least $r-1$ of them, a contradiction.

Thus we may assume the vertices of $e \cap V(C)$ are all consecutive along $C$. Then we have \[r-2+r+1 \leq s,\] and $r \leq (s+1)/2$. If $s \leq n-4$, we get an immediate contradiction. If $s = n-2$, then similarly to above, $e$ prohibits $r+1$ vertices of $C$ from being contained in any $f \in E(H')$. Thus there are only $r-1$ vertices remaining in $V(C)$ that can be contained in any edge of $H'$, but there are at least four edges of $H'$ distinct from $e$, a contradiction.

Finally, we have that $s=n-3$, and there is some $v' \notin V(C)$ distinct from $u_1$ and $v$. In this case, there are at least $(n+3)-(n-3) = 6$ edges of $H'$, so we may assume without loss of generality that $v \in f \in E(H')$ for some $f \neq e$. However, $e$ prohibits $r+1$ of the at most $2r-1$ vertices of $C$ from being contained in $f$, a contradiction.

\end{proof}

\subsection{The case of $\ell=1$ and $r<t$}

\begin{lemma}\label{ell3}
Let $n$, $k$, and $r$ be positive integers such that $n \geq k$ and $r < t$. If $H$ is an $r$-uniform hypergraph with at least $k$ edges such that
$\delta(H) \geq {t \choose r-1} + 1$ and $c(H) < k$, then $\ell = |V(P)| \geq 2$.
\end{lemma}
%For the $r<t$ case, we need only consider when $k=n$ and hence $\delta(H) \geq {t \choose r-1} + 1.$ 
\begin{proof}Suppose $\ell = 1$. Since every edge in $H'$ contains at most one vertex outside of $C$, $N_{H'}(u_1) \subseteq V(C)$.

By Claim~\ref{noconsecutive}, $|N_{H'}(u_1)| \leq \lfloor s/2 \rfloor \leq t$. Let $b_1$ be the number of edges in $E(C)$ containing $u_1$. By Claim~\ref{b12}, we must have $b_1 \geq 2$. 

Corollary~\ref{cycleneighbors} additionally gives that if $2 \leq b_1 \leq s-1$, then $|N_{H'}(u_1)| \leq \lceil \frac{s-1-b_1}{2} \rceil$, and if $b_1 = s$, then $|N_{H'}(u_1)| = 0$.

Notice that \[{t \choose r-1} - {t-1 \choose r-1} \geq {t \choose 2} - {t-1 \choose 2} = t-1
\] for $t \geq r+2$. Similarly, if $t=r+1$, then ${t \choose r-1} - {t-1 \choose r-1} = {t \choose 2} - (t-1) \geq t-1$. 
Thus if $b_1 \leq t-1$, we have \[d(u_1) \leq b_1 + {|N_{H'}(u_1)| \choose r-1} \leq t-1 + {t-1 \choose r-1} \leq {t \choose r-1},\] a contradiction.  Therefore we may assume $b_1 \geq t$. This gives that $|N_{H'}(u_1)| \leq \lceil \frac{s-t-1}{2} \rceil \leq \lceil \frac{t}{2} \rceil$. 

We have that \[{t \choose r-1} - {\lceil \frac{t}{2} \rceil \choose r-1} \geq {t \choose 2} - {\lceil \frac{t}{2} \rceil \choose 2} \geq n-1 \geq b_1
\] whenever $\lceil \frac{t}{2} \rceil \geq r+1$ and $t \geq 7$. If $\lceil \frac{t}{2} \rceil =r$, then we instead have ${t \choose r-1} - {\lceil \frac{t}{2} \rceil \choose r-1} \geq {t \choose 2} - \lceil \frac{t}{2} \rceil \geq n-1$ when $t \geq 7$. If $\lceil \frac{t}{2} \rceil \leq r-1$, then we have $d(u_1) \leq b_1 + 1 \leq n \leq {t \choose r-1}$ whenever $t \geq 6$. Hence for $t \geq 7$, we have $d(u_1) < \delta(H)$, a contradiction.

For the remaining values of $t$, we consider whether or not $|N_{H'}(u_1)| = 0$. First suppose we have $|N_{H'}(u_1)| \geq r-1$ and hence $\lceil \frac{s-b_1-1}{2} \rceil \geq r-1$. When $t=4$, we need $s\in \{8,9\}$, $b_1\in \{4,5\}$ (since $b_1 \geq t$) to have $\lceil \frac{s-b_1-1}{2} \rceil \geq r-1$. In every case we have $|N_{H'}(u_1)| = r-1$, but then $d(u_1) \leq 5+1  < 7 \leq \delta(H)$. When $t=5$, we have $s \leq 11$ and so we need $b_1 \leq 7$ to have $\lceil \frac{s-b_1-1}{2} \rceil \geq r-1 \geq 2$. Hence $d(u_1) \leq 7+ 1 < 11 \leq \delta(H)$. When $t=6$, we have $6 \leq b_1\leq s \leq 13$, so $\lceil \frac{s-b_1-1}{2} \rceil \leq 3$ and hence we are done if $r \geq 5$. If $\lceil \frac{s-b_1-1}{2} \rceil = r-1$, then $d(u_1) \leq b_1 + 1 < \delta(H)$. If $\lceil \frac{s-b_1-1}{2} \rceil = r = 3$, then we must have $b_1 \leq 6$, so $d(u_1) \leq 6 + 3 < \delta(H)$, a contradiction.

For the final case of $|N_{H'}(u_1)| = 0$, we prove a brief claim.

\begin{claim}\label{cycleneighbors2}
If $|N_{H'}(u_1)| = 0$, then $b_1 \leq s-r+2$.
\end{claim}

\begin{proof}
Suppose that $b_1 \geq s-r+3$. Notice that we must have $E(H') \neq \emptyset$ because there are at least $k > s$ edges. Let $e \in E(H')$, and notice that $|e\cap V(C)| \geq r-1 \geq 2$. Thus there must exist $v_i, v_j \in e$ such that $u_1 \in e_i, e_j$ or $u_1 \in e_{i-1}, e_{j-1}$ because $u_1$ is in all but at most $r-3$ edges of $C$. However, we can then consider the cycle \[ v_1, e_1, v_2, \ldots, e_{i-1}, v_i, e, v_j, e_{j-1}, v_{j-1}, \ldots, e_{i+1}, v_{i+1}, e_i, u_1, e_j, v_{j+1}, e_{j+1}, \ldots, e_{s-1}, v_s, e_s, v_1,
\]
which is longer than $C$, a contradiction.
\end{proof}

If we do have $|N_{H'}(u_1)| = 0$, then Claim~\ref{cycleneighbors2} gives that $b_1 \leq s-r+2 \leq n-r+1$. Then $d(u_1) \leq n-r+1 < \delta(H)$ except in the case $t=4, r=3, b_1\geq 7$, which we handle separately.

\medskip
{\bf Case 1:} $s = n-1 \in \{8,9\}$. Therefore $s - b_1 \leq 2$.
Let $e \in E(H')$, and notice that $e \subseteq V(C)$ because $|N_{H'}(u_1)| = 0$. As in the case of $\ell=1$, $r\geq t$, $e$ prohibits some edges of $C$ from containing $u_1$. That is, if $v_i, v_j \in e$, then $u_1$ cannot be contained in both $e_i$ and $e_j$ and cannot be contained in both $e_{i-1}$ and $e_{j-1}$. If $e$ is not all consecutive, then $e$ prohibits at least 3 edges of $C$ from containing $u_1$. This contradicts that $s-b_1 \leq 2$. If $e$ is all consecutive, say $e = \{v_i, v_{i+1}, v_{i+2}\}$, notice that if $u_1 \in e_i$, then we must have $u_1 \notin e_{i-1}, e_{i+1}, e_{i+2}$, reaching the same contradiction. Thus we have $u_1 \notin e_i$ and similarly $u_1 \notin e_{i+1}$. Consider the cycle formed by swapping the roles of $e$ and $e_i$. Then $e_i$ must prohibit at least one additional edge of $C$ from containing $u_1$, reaching the same contradiction again.

\medskip
{\bf Case 2:} $s=8$, $n=10$, $b_1 = 7$.
If any edge of $E(H')$ is contained fully in $V(C)$, then we follow the same arguments as Case 1 to reach a contradiction. Thus we may assume every edge of $E(H')$ contains the unique vertex $x \neq u_1$ outside $C$. Let $e_j$ be the edge of $C$ which does not contain $u_1$. For any edge $e \in E(H')$, we must have $v_i, v_{i+1} \in e$, as otherwise $e$ will prohibit at least two edges of $C$ from containing $u_1$. However, there are at least two such edges $e, e' \in E(H')$, and this gives $e = e'$, a contradiction.
\end{proof}

\section{Proof of Theorem~\ref{main3}(c)}

\begin{proof}
As in previous sections, consider a best pair $(C,P)$ with $C = v_1, e_1, v_2, \ldots, e_{s-1}, v_s, e_s, v_1$ and $P = u_1, f_1, u_2, \ldots, f_{\ell-1}, u_\ell$.  We use the same notation of $H_C, H_P, H'$, and additionally define the following. For a vertex $v$ of a hypergraph $F$, $F\{v\}$ will denote the set of the edges of $F$ containing $v$. 

 By Lemmas~\ref{ell2} and~\ref{ell3}, $\ell \geq 2$. By Lemma~\ref{le1}, $s \geq t+2$. Therefore $\ell \leq n - s \leq 2t + 2 - (t+2) = t$. 

Recall for $j \in \{1, \ell\}$, $B_j=H_C\{u_j\}$, and set $b_j=|B_j|$. By symmetry, we may assume $b_\ell\geq b_1$.
%If $i<i'\leq i+\ell-1$, $e_i\in H_C\{u_1\}$ and  $e_{i'}\in H_C\{u_\ell\}$, then replacing path $v_i,e_i,\ldots,v_{i'},e_{i'},v_{i'+1}$ in $C$ with the path
%$v_i,e_i,u_1,f_1,\ldots,f_{\ell-1},u_\ell,e_{i'},v_{i'+1}$ creates a cycle longer than $C$, a contradiction. So, by the dual version of Lemma~\ref{verc},
By Claim~\ref{distance} and Lemma~\ref{verc} either
\begin{equation}\label{n2'}
b_1\leq (s+2)/2-\ell,\end{equation}
or
%\begin{equation}\label{n12} b_1\leq s/\ell-1,\end{equation} or 
\begin{equation}\label{n13}
\mbox{$B_1=B_\ell$  and }\; b_1\leq s/\ell.
\end{equation}
%If~\eqref{n2} holds, then since $b_1\leq b_\ell$, 
%$$b_1\leq \frac{(2t+2-\ell)-2\ell+3}{2}=t+2-\ell+\frac{1-\ell}{2}<t+2-\ell,$$
%contradicting~\eqref{n1}. So by~\eqref{n3}, $b_1\leq s/\ell\leq (2t+2)/\ell-1$.

Recall that by the maximality of $V(P)$ all edges in $H'$ containing $u_1$ or $ u_\ell$ are contained in $V(C)\cup V(P)$.

%\medskip

For $j\in \{1,\ell\}$, let $A_j=N_{H'}(u_j)\cap V(C)$ and $a_j=|A_j|$. By Claim~\ref{noconsecutive}, $A_j$ contains no consecutive vertices of $C$ for $j \in \{1, \ell\}$.

{\bf Case 1:} $A_1=\emptyset$. Then all edges  in $H'$ containing $u_1$  are contained in $ V(P)$.

{\bf Case 1.1:} $r=t$. Since $\ell\leq t$, the only possibility of an edge $g\in E(H')$ containing $u_1$ is that $\ell=t$ and $g=V(P)$. But then we can switch $g$ with $f_1$, contradicting Part (iv) of choosing $(C,P)$. Thus $N_{H'}(u_1)=\emptyset$. Then
\begin{equation}\label{n14}
b_1\geq \delta(H) - |E(P)| \geq (t+1)-(\ell-1)=t-\ell+2.
\end{equation}
So, if~\eqref{n2'} holds, then since $s\leq n-\ell\leq 2t+2-\ell$, $b_1\leq \frac{2t+2}{2}-\ell$, contradicting~\eqref{n14}.

%If~\eqref{n12} holds, then comparing with~\eqref{n14} we get $t-\ell+2\leq (2t+2-\ell)/\ell-1$, i.e. $\ell(t-\ell+4)\geq 2t+2$,
%which does not hold for $2\leq \ell\leq t$. 

If~\eqref{n13} holds, then  comparing with~\eqref{n14} we get  $t-\ell+2\leq (2t+2-\ell)/\ell$, which is equivalent to
$\ell(t-\ell+3)\geq 2t+2$.
This can hold only when $\ell=2$ and $s=2t$. In this case $b_1=t$ and $B_\ell=B_1$. Since an edge in $B_\ell$ cannot be next to an edge in $B_1$ on $C$, we may assume that  $B_1=B_\ell=\{e_1,e_3,\ldots,e_{2t-1}\}$.
 Since $n=s+\ell=s+2$, $f_1$ contains a vertex of $C$, say $v_1$. But then we get a longer cycle by replacing in $C$  path
 $v_1,e_1,v_2$ with path $v_1,f_1,u_1,e_1,v_2$, a contradiction.

{\bf Case 1.2:} $3\leq r\leq t-1$. The number of edges in $H'$ containing $u_1$ and contained in $V(P)$ is at most $\binom{\ell-1}{r-1}$.
So, 
\begin{equation}\label{n15}
b_1 \geq 1+\binom{t}{r-1}-\binom{\ell-1}{r-1}-(\ell-1)\geq 1+\binom{t}{2}-\binom{\ell-1}{2}-\ell+1=\frac{(t+\ell-2)(t-\ell+1)}{2}-\ell+2.
\end{equation}
If~\eqref{n2'} holds, then since $s\leq 2t+2-\ell$, we get
$$\frac{(t+\ell-2)(t-\ell+1)}{2}-\ell+2\leq \frac{2t+4-\ell}{2}-\ell,$$
which is not true for $2\leq \ell\leq t$ when $t\geq 4$.

If~\eqref{n13} holds, then  we get
\begin{equation}\label{bfr}
\frac{(t+\ell-2)(t-\ell+1)}{2}-\ell+2\leq  \frac{2t+2-\ell}{\ell}.
\end{equation}
This does not hold in the range $2\leq \ell\leq t-1$ when $t\geq 4$. Suppose now $\ell=t\geq 4$.
If  all $\binom{\ell-1}{r-1}$ $r$-subsets of 
$V(P)$ containing $u_1$ are in $H'$, then we can replace $f_1$ with $\{u_1,\ldots,u_r\}$ contradicting Part (iv) of choosing $(C,P)$.
Thus, in this case instead of~\eqref{n15}, we have $b_1\leq \frac{(t+\ell-2)(t-\ell+1)}{2}-\ell+3$ and so instead of~\eqref{bfr}, we have
$$\frac{(t+\ell-2)(t-\ell+1)}{2}-\ell+3\leq  \frac{2t+2-\ell}{\ell},$$
which is not true for $\ell=t\geq 3$. This finishes Case 1.

\medskip
{\bf Case 2:} $A_1\neq\emptyset$ and $B_\ell\neq \emptyset$. 
If $v_i\in A_1$, say $v_i\in g\in E(H'\{u_1\})$, $e_j\in B_\ell$, $j\geq i$ and $j-i\leq \ell-2$, then by replacing in $C$ path $v_i,e_i,v_{i+1},\ldots,e_j,v_{j+1}$
  with the path
$v_i,g,u_1,f_1,\ldots,f_{\ell-1},u_\ell,e_{j},v_{j+1}$ creates a cycle longer than $C$, a contradiction. Thus $A_1\cap B_\ell=\emptyset$,
each interval of $C-A_1$ contains a vertex not covered by $B_\ell$, and each such interval containing an edge in $B_\ell$ has at least $2(\ell-1)$ such vertices. Since the edges in $B_\ell$ cover at least $b_\ell+1$ vertices, we get
\begin{equation}\label{n7'}
(2a_1-1)+2(\ell-1)+(b_\ell+1)\leq s\leq 2t+2-\ell.
\end{equation}
Since $\ell\geq 2$ and by the case $b_\ell\geq 1$,~\eqref{n7'}  yields $2a_1+2\ell-1\leq 2t$, so by integrality
\begin{equation}\label{n7}
t\geq a_1+\ell.
\end{equation}
If $r=t$,~\eqref{n7} yields that $H'\{u_1\}$ contains only one edge, namely, $g=A_1\cup V(P)$.
But then we can switch $g$ with $f_1$ and still have the best pair $(C,P')$ where $P'$ is obtained from $P$ by deleting $f_1$ and adding $g$ instead.
So, there is a vertex $v_i\in (f_1\cap V(C))-A_1$. This is one more vertex that is not next to any $v_j\in A_1$ and is at distance in $C$ at least $\ell$ from $B_\ell$. Thus in this case instead of~\eqref{n7'} we get $(2a_1+1)+2(\ell-1)+(b_\ell+1)\leq s$ and hence $t\geq a_1+\ell+1$, a contradiction.

Suppose now $3\leq r\leq t-1$. Then, since $b_\ell\geq b_1$,
$$1+\binom{t}{r-1}\leq d(u_1)= d_{H'}(u_1)+b_1+d_{H_P}(u_1)\leq \binom{a_1+(\ell-1)}{r-1}+b_\ell+(\ell-1).$$
So,
$$\frac{t(t-1)-(a_1+\ell-1)(a_1+\ell-2)}{2}=\binom{t}{2}-\binom{a_1+\ell-1}{2}\leq \binom{t}{r-1}-\binom{a_1+\ell-1}{r-1}\leq b_\ell+\ell-2.$$
Plugging in the upper bound on $b_\ell+\ell-2$ from~\eqref{n7'} and rewriting $\frac{t(t-1)-(a_1+\ell-1)(a_1+\ell-2)}{2}$ as
$\frac{(t+a_1+\ell-2)(t-a_1-\ell+1)}{2}$, we obtain
\begin{equation}\label{n8}
\frac{(t+a_1+\ell-2)(t-a_1-\ell+1)}{2}\leq 2(t-a_1-\ell+1).
\end{equation}
Since by~\eqref{n7}, $t-a_1-\ell+1>0$,~\eqref{n8} simplifies to $t+a_1+\ell-2\leq 4$. Since $t\geq r+1\geq 4$, $a_1\geq 1$ and $\ell\geq 2$, this is impossible.

\medskip
{\bf Case 3:} $A_1\neq\emptyset$,  $B_\ell=B_1= \emptyset$, and 
$A_\ell\neq A_1$.
%none of the edges in $H'$ contains both $u_1$ and $u_\ell$. Similarly to $A_1$ and $a_1$, define $A_\ell=N_{H'}(u_\ell)\cap V(C)$ and $a_\ell=|A_\ell|$. 
By  Case~1, $a_1>0$ and 
$a_\ell>0$. 

If $i<i'\leq i+\ell$, and there are distinct $g_1,g_\ell\in E(H')$ such that $\{v_i,u_1\}\subset g_1$ and  $\{v_{i'},u_\ell\}\subset g_\ell$,
 then replacing path $v_i,e_i,\ldots,v_{i'}$ in $C$ with the path
$v_i,g_1,u_1,f_1,\ldots,f_{\ell-1},u_\ell,g_\ell,v_{i'}$ creates a cycle longer than $C$, a contradiction. 

By Claim~\ref{noconsecutive}, $A_1\cup A_\ell$ does not contain consecutive vertices of $C$.
We may assume that $a_1\leq a_\ell$. Then since $A_\ell\neq A_1$, $A_\ell-A_1\neq \emptyset$.
So, by  Lemma~\ref{verc2}, 
\begin{equation}\label{n2''}
a_1\leq (s+2)/2-\ell-1\leq (2t +2-\ell)/2-\ell\leq t-\ell.\end{equation}
%or  \begin{equation}\label{n13'}
%\mbox{$A_1=A_\ell$  and }\; a_1\leq s/(\ell+1). \end{equation}
Also
\begin{equation}\label{n21}
d_{H'}(u_1)\geq d_H(u_1)-b_1-(\ell-1)\geq  1+\binom{t}{r-1}-0-\ell+1=2+\binom{t}{r-1}-\ell.
\end{equation}

{\bf Case 3.1:} $r=t$.  Then each edge  $g\in H'\{u_1\}$ has at least $t-\ell$ vertices in $V(C)$ 
with equality only when $V(P)\subset g$.
By~\eqref{n21},  $d_{H'}(u_1)\geq 2$.  Hence  $a_1\geq (t+1)-\ell$. 
This contradicts~\eqref{n2''}.

 %If~\eqref{n13'} holds, then we get $(t+2)-\ell\leq s/(\ell+1)$ which does not hold for $2\leq \ell\leq t$
%when $s\leq 2t+2-\ell$.

{\bf Case 3.2:} $3\leq r\leq t-1$. Then 
$d_{H'}(u_1)\leq \binom{a_1+\ell-1}{r-1}$. So,  by~\eqref{n2''},
$d_{H'}(u_1)\leq \binom{t-1}{r-1}$, and together with~\eqref{n21}, we get 
$$2+\binom{t}{r-1}-\ell\leq \binom{t-1}{r-1},$$
which is not true when $2\leq r-1\leq t-2$ and $\ell\leq t$. 

\medskip
{\bf Case 4:} $A_1\neq\emptyset$,  $B_\ell=B_1= \emptyset$, and  
$A_\ell=A_1$. Let $A_1=\{x_1,\ldots,x_{a_1}\}$ with vertices in clockwise order on $C$.

{\bf Case 4.1:}  Between any $x_j$ and $x_{j+1}$ there are at least $\ell$ vertices. Then $(\ell+1)a_1\leq s$.
If $a_1\geq 2$, then~\eqref{n2''} holds, and we repeat the argument of Case~3. So suppose
 $a_1=1$ and $A_\ell=A_1=\{v_1\}$.
 Since 
 \begin{equation}\label{n17}
d_{H'}(u_1)\geq 1+\binom{t}{r-1}-(\ell-1)\geq 1+\binom{t-1}{r-1}
\end{equation}
 and each edge in $H'\{u_1\}$ is contained in
$V(P)+v_1$, $\ell=t$ and some edge $g\in H'\{u_1\}$ contains $u_\ell$. Also, by degree condition, some edge $f\in H\{u_1\}$ is not contained in $V(P)+v_1$. By the case, this is some $f_j$. By the symmetry between $u_1$ and $u_\ell$, we may assume
$j\leq \ell/2$. Since $H$ contains path $P_j=u_{j+1},f_{j+1},\ldots,u_\ell,g,u_1,f_1,\ldots,u_j$, the edge $f_j$ is contained in $V(C)\cup C(P)$, and hence
$f_j$ contains some $v_i$ for $i\neq 1$. By symmetry, we may assume $i\leq s/2+1=t/2+2$.

If all $r$-element subsets of $V(P)$ containing $u_1$ are edges in $H$, then by Rule (iv) of the choice of $(C,P)$, one of them is $f_1$.
Thus, by~\eqref{n17}, there is $g_1\in H'\{u_1\}-g$ not contained in $V(P)$ and hence containing $v_1$. When we replace path
$v_1,e_1,v_2,\ldots,v_i$ in $C$ with path $$v_1,g_1,u_1,g,u_\ell,f_{\ell-1},u_{\ell-1},\ldots,u_{j+1},f_j,v_i,$$ we first delete the $i-2$ internal vertices of the former path and  then add $t-j+1$ vertices of the latter. So, the length of the cycle will be at least
$$s-(i-2)+(t-j+1)\geq s-t/2+t/2+1>s,$$
a contradiction.

{\bf Case 4.2:}  There are  indices $j$ such that between $x_{j}$ and $x_{j+1}$ there are at most $\ell-1$ vertices. Since $A_\ell=A_1$, for each such $j$ there is a $g_j\in E(H'\{u_1\})\cap E(H'\{u_\ell\})$ containing $x_{j}$ and $x_{j+1}$
and no other edge in $E(H'\{u_1\})\cup E(H'\{u_\ell\})$ contains any of $x_{j}$ and $x_{j+1}$. Suppose that we have exactly $m$ such edges $g_j$ and that $A$ is the set of the 
 vertices in $A_1$ that are not in such pairs $(x_{j},x_{j+1})$. Let $a=|A|$.
 By these definitions and remembering that $A_1$ is independent,
$a_1\geq a+2m$, 
 \begin{equation}\label{n19}
d_{H'}(u_1)\leq m+\binom{|V(P)\cup A|-1}{r-1}\leq m+\binom{\ell+a-1}{r-1}\leq \binom{\ell+a+m-1}{r-1},
\end{equation}
and if $a+m\geq 2$, then
\begin{equation}\label{n20}
2t\geq s\geq (\ell+1)(a+m)+m.
\end{equation} 
Since our case is that $m\geq 1$, the negation of $a+m\geq 2$ means $a=0$ and $m=1$. In this case only one edge, say $g_0$ in $H'\{u_1\}$ intersects $V(C)$.  By~\eqref{n21}, $H'\{u_1\}$ contains another edge, say $g_1$ that must be contained in $V(P)$. This yields
$\ell=t$ and $g_1=V(P)$. But then we can switch $g_1$ with $f_1$ contradicting Rule (iv) of the choice of $(C,P)$. Thus,
$a+m\geq 2$ and so~\eqref{n20} holds. Since $m\geq 1$ and $a+m\geq 2$,~\eqref{n20} yields $2t\geq 2(a+m)+(\ell-1)2 +1$, 
and so $t>a+m+\ell-1$. This means
\begin{equation}\label{n22}
t\geq a+m+\ell.
\end{equation} 
Plugging~\eqref{n22} into~\eqref{n19} and comparing with~\eqref{n21}, we get
$$\binom{t}{r-1}-\binom{t-1}{r-1}\leq \ell-2,$$
which does not hold for $\ell\leq t$ when $r\leq t$.
\end{proof}

\section{Proof of Theorem~\ref{main41}}\label{large}

\begin{proof}
Let $k$ be the smallest integer at least $n/2$ for which the theorem does not hold. Let $H$ be an $n$-vertex $r$-graph with $\delta(H)\geq \lceil k/2 \rceil$ such that $H$ has no cycle of length $k$ or longer.

Choose a best pair $(C,P)$ with notation as in the previous two sections. By Lemma~\ref{le2}, $k>(n+1)/2$. Moreover, by Lemma~\ref{ell1}, $\ell \geq 2$.

Since the theorem holds for $k'<k$, $s=k-1$. Also by the maximality of $\ell$, each edge in $H'$ containing $u_1$ or $u_\ell$ is contained in $V(C)\cup V(P)$ and cannot have two consecutive vertices of $C$.

\medskip
{\bf Case 1:} $\ell\geq (1+k)/2$.

\medskip
{\bf Case 1.1:} There are distinct $v_i$ and $v_j$ in $V(C)$ such that $v_i\in f_1$ and $v_j\in f_{\ell-1}$. By symmetry, we may assume that $i=1$ and
$j\leq (s+1)/2$. By the maximality of $s$, the path $v_1,f_1,u_2,f_2,\ldots, u_{\ell-1},f_{\ell-1},v_j$ is not longer than the path
$v_1,e_1,\ldots,e_{j-1},v_j$. This means $\ell-2\leq j-2$.
Plugging in the inequalities for $\ell$ and $j$, we get 
$$(1+k)/2\leq  (s+1)/2\leq k/2,$$
a contradiction

\medskip

{\bf Case 1.2:} Case 1.1 does not hold. Then either $f_1$ or $f_{\ell-1}$ contains at most one vertex in $C$, so $s+r\leq n+1$. By Lemma~\ref{le2}, this is only possible when $r=n/2$ and $s=1+n/2$.
Since $r+s>n$,
each of $f_1$ and $f_{\ell-1}$ has a vertex in $C$. Since Case 1.1 does not hold, this is the same vertex, say $v_1$. Moreover, each of 
$f_1$ and $f_{\ell-1}$ must contain $V(G)-V(C)$. But then $f_1=f_{\ell-1}$ and so $\ell=2$. By the case, $2\geq (k+1)/2$, i.e., $k\leq 3$, so
$3>(n+1)/2$, thus $n\leq 4$, and $r \leq n/2 \leq 2$, a contradiction.

{\bf Case 2:} $2\leq \ell \leq k/2$. Since $s\geq (n+1)/2$, $\ell\leq n-s<n/2\leq r$. So, $r-\ell\geq 1$.

{\bf Case 2.1:}  There is an edge $g\in E(H')$ containing $u_1$. By the maximality of $|V(P)|$, $g\subset V(C)\cup V(P)$.
So $|g\cap V(C)|\geq r-\ell$. Since no vertices of $g$ are consecutive on $C$, the number of vertices in the largest interval of $C$ between
vertices of $g$ is at most 
\begin{equation}\label{rn1}
s-2(r-\ell)+1\leq (n-\ell)-2r+2\ell+1\leq \ell+1.
\end{equation}
This means, the distance on $C$ from any of its vertices to $g$ is at most $1+\ell/2$. 

{\bf Case 2.1.1:} Some $e_i$ contains $u_\ell$, say $i=1$. If some $v_{j}\in g$ and $j\leq \ell+1$, then we can replace the path $v_1,e_1,v_2,\ldots,v_j$ in $C$ with the path 
$v_1,e_1,u_\ell,f_{\ell-1},u_{\ell-1}, \ldots,u_1,g,v_j$, and get a longer cycle. Thus the interval of $C$ between two vertices of $g$ that contains $e_1$ has at least $2+2(\ell-1)=2\ell$ vertices, contradicting~\eqref{rn1}.

{\bf Case 2.1.2:} None of $e_i$ contains $u_\ell$. Since $d(u_\ell)\geq k/2\geq \ell$ and $P$ has only $\ell-1$ edges,
there is an edge $g'\in E(H')$ containing $u_\ell$. So, by symmetry we may assume that none of $e_i$ contains $u_1$.

Suppose first $g'\neq g$.
Since the distance on $C$ between any  vertex of $g\cap V(C)$ 
 and any  vertex of $g'\cap V(C)$ is either $0$ or at least $1+\ell$,
all vertices of $g'\cap V(C)$ must belong to $g$, and the distance on $C$ between any two vertices of $g'$ is at least $1+\ell$. 
By symmetry, we get $g\cap V(C)=g'\cap V(C)$. Since $g\neq g'$, the edges must differ in $V(P)$. In particular, $|g\cap V(P)|\leq \ell-1$, and hence $|g\cap V(C)|\geq r-\ell+1$.
But then
\begin{equation}\label{rn2}
n\geq s+\ell\geq (1+\ell)(r-\ell+1)+\ell.
\end{equation}
The minimum of the polynomial $F(\ell)=-\ell^2+(r+1)\ell+r+1$ in the RHS of~\eqref{rn2} is attained when $\ell$ is extremal. 
We have $F(2)=F(r-1)=-1+3r$, which is greater than $n$ when $r\geq \max\{3,n/2\}$. 

Suppose now only $g$ is an edge in $H'$ containing $u_\ell$. We have that $H\{u_1\} = H\{u_\ell\} = E(P) \cup \{g\}=:L$. 
Moreover, for any $u_i \in V(P)$, the path $P_i^1 = u_i, f_{i-1}, \ldots, f_1, u_1, f_i, u_{i+1}, \ldots, f_{\ell -1}, u_\ell$ has the same length, vertices, and edges as $P$. We conclude that $(C, P_i^1)$ is also best pair, and so we may assume that $H\{u_i\} = L$ for all $1 \leq i \leq \ell$. Therefore $V(P) \subset g$ and $\ell = r-1$. 
% Then $\ell=r-1$ and $g\supset V(P)$. 

Moreover, for every $1\leq j\leq \ell-1$, the path
$P(j)=u_{j+1},f_{j+1},\ldots,u_\ell,g,u_1,f_1,\ldots,u_j$ has the same vertex set as $P$, and its ends, $u_j$ and $u_{j+1}$ belong to edge $f_j$ not used in $P(j)$. The pair $(C, P(j))$ is also a best pair since $V(P) \subset g$. As above we conclude that $f_j\supset V(P)$.
In particular, $\ell=|L|=\lceil k/2\rceil$. Also, each edge in $L$ has exactly one vertex on $C$ and these vertices are distinct. Since
$\ell\geq k/2>s/2$, some vertices of edges in $L$ are consecutive on $C$. By symmetry, we may assume $v_s\in g$ and $v_1\in f_1$.
Then replacing edge $e_s$ in $C$ by path $v_s,g,u_1,f_1,v_1$, we get a cycle longer than $C$.

{\bf Case 2.2:}  No edge in $E(H')$ contains $u_1$ or $u_\ell$. 
%\medskip
 Recall $B_1$ (respectively, $B_\ell$) is the set of edges $e_i$ that contain
$u_1$ (respectively, $u_\ell$). Then for $j\in \{1,\ell\}$,  $|B_j|\geq \delta(H)-|E(P)| \geq \lceil k/2\rceil - \ell +1$. If $B_1$ and/or $B_\ell$ has size greater than $\lceil k/2\rceil-\ell+1$, then we can delete some edges to make both have exactly $\lceil k/2\rceil-\ell+1$ edges and be different from each other.

By Claim~\ref{distance}, for any distinct $e_i \in B_1$ and $e_j \in B_\ell$, $|i - j| \geq \ell$. So, if $B_1\neq B_2$, then we apply Lemma~\ref{verc} to $B_1, B_\ell$ and $q=\ell$ to obtain 
$s\geq 2(\lceil k/2\rceil-\ell+1)+2(\ell-1)\geq k$, a contradiction. Thus  $B_1=B_2$ and $|B_1|=
\lceil k/2\rceil-\ell+1$. It follows that our original  $B_1$ and $B_2$ are the same. 
For this, we need $\{u_1,u_\ell\}\subset f_i$ for all $1\leq i\leq \ell-1$ and for $u\in \{u_1,u_\ell\}$,
\begin{equation}\label{o4}
\mbox{\em the set of edges containing $u$ is $B_1\cup \{f_1,\ldots,f_{\ell-1}\}$.}
\end{equation}
If $f_1$ contains a vertex $u\in V(G)-V(C)-V(P)$, then $u$ can play the role of $u_1$, and hence~\eqref{o4} holds, as well. Also, for each $1\leq  j<\ell$, since $u_1\in f_j$, the path $P_j^1=u_j,f_{j-1},u_{j-1},\ldots,u_1,f_j,u_{j+1},f_{j+1},\ldots,u_\ell$ can play role of $P$. It follows that
\eqref{o4} holds for  $u=v_j$ and hence for all $u\in f_{j-1}$.

By symmetry, let $e_1\in B_1$. By the above, $e_1$ contains $\{u_1,\ldots,u_\ell\}$, all vertices in $f_1-V(C)-V(P)$, and $v_i, v_{i+1}$. Since $|e_1|=r=|f_1|$, the edge
$f_1$ has at least two vertices in $C$. These vertices must be at distance in $C$ at least $\ell-1$ from any edge in $B_2$. It follows that
$$s\geq \ell (k/2-\ell+1)+ (\ell-2)+2=\ell(k/2-\ell+2).$$
For $2\leq \ell \leq k/2$, the RHS is at least $k$, a contradiction.\end{proof}

\end{document}